\documentclass[12pt,leqno]{article}

\usepackage{graphicx, fullpage}
\usepackage{amsmath,amssymb,amsthm,amscd, bm}
\usepackage{fancyhdr}
\usepackage[mathscr]{eucal}
\usepackage{amsfonts}
\usepackage[T1]{fontenc}
\usepackage{xspace}
\usepackage{esint}
\begin{document}
\parskip=6pt

\theoremstyle{plain}
\newtheorem{prop}{Proposition}
\newtheorem{lem}[prop]{Lemma}
\newtheorem{thm}[prop]{Theorem}
\newtheorem{cor}[prop]{Corollary}
\newtheorem{defn}[prop]{Definition}
\theoremstyle{definition}
\newtheorem{example}[prop]{Example}
\theoremstyle{remark}
\newtheorem{remark}[prop]{Remark}
\numberwithin{prop}{section}
\numberwithin{equation}{section}
%\numberwithin{example}{section}

%%%%Put a brace on the right, not left, to group cases
\newenvironment{rcases}
  {\left.\begin{aligned}}
  {\end{aligned}\right\rbrace}

%________________________________________
%original with \bf Example 3.2...
%\theoremstyle{plain}
%\newtheorem {thm}{Theorem}[section]
%\newtheorem {lem}[thm]{Lemma}
%\newtheorem {cor}[thm]{Corollary}
%\newtheorem {defn}[thm]{Definition}
%\newtheorem {prop}[thm]{Proposition}
%\numberwithin{equation}{section}
%________________________________________

%\newtheorem{proof}{Proof}
%\newtheorem{proof2}{Proof of Theorem 1.2}
%\newtheorem{proof3}{Proof of Theorem 1.3}
%\newtheorem{proof4}{Proof of Theorem 1.4}
%\numberwithin{proof2}{Proof of Theorem 1.2}
\def\cal{\mathcal}
\newcommand{\cF}{\cal F}
\newcommand{\cG}{\cal G}
\newcommand{\cA}{\cal A}
\newcommand{\cB}{\cal B}
\newcommand{\cC}{\cal C}
\newcommand{\cO}{{\cal O}}
\newcommand{\cE}{{\cal E}}
\newcommand{\cU}{{\cal U}}
\newcommand{\cM}{{\cal M}}
\newcommand{\cD}{{\cal D}}
\newcommand{\cK}{{\cal K}}
\newcommand{\cZ}{{\cal Z}}
\newcommand{\cH}{{\cal H}}
\newcommand{\cL}{{\cal L}}
\newcommand{\cS}{{\cal S}}
\newcommand{\cW}{{\cal W}}

\newcommand{\fQ}{\frak{Q}}

\newcommand{\bC}{\mathbb C}
\newcommand{\bP}{\mathbb P}
\newcommand{\bN}{\mathbb N}
\newcommand{\bA}{\mathbb A}
\newcommand{\bR}{\mathbb R}
\newcommand{\bZ}{\mathbb Z}
\newcommand{\oP}{\overline P}
\newcommand{\oQ}{\overline Q}
\newcommand{\oR}{\overline R}
\newcommand{\oS}{\overline S}
\newcommand{\oc}{\overline c}
\newcommand{\bp}{\mathbb p}
\newcommand{\oD}{\overline D}
\newcommand{\oE}{\overline E}
\newcommand{\oC}{\overline C}
\newcommand{\of}{\overline f}
\newcommand{\ou}{\overline u}
\newcommand{\oU}{\overline U}
\newcommand{\oj}{\overline j}
\newcommand{\oV}{\overline V}
\newcommand{\ov}{\overline v}
\newcommand{\ow}{\overline w}
\newcommand{\oy}{\overline y}
\newcommand{\oz}{\overline z}
\newcommand{\pa}{\partial}
\newcommand{\op}{\overline \partial}
\newcommand{\ochi}{\overline \chi}

\newcommand{\hg}{\hat G}
\newcommand{\hM}{\hat M}

\newcommand{\tpr}{\widetilde {\text{pr}}}
\newcommand{\tB}{\widetilde B}
\newcommand{\tx}{\widetilde x}
\newcommand{\ty}{\widetilde y}
\newcommand{\txi}{\widetilde \xi}
\newcommand{\teta}{\widetilde \eta}
\newcommand{\tna}{\widetilde \nabla}
\newcommand{\tth}{\widetilde \theta}
\newcommand{\tva}{\widetilde \varphi}

\newcommand{\diml}{\text{dim}}
\newcommand{\var}{\varepsilon}
\newcommand{\End}{\text{End }}
\newcommand{\loc}{\text{loc}}
\newcommand{\Symp}{\text{Symp}}
\newcommand{\Sympo}{\text{Symp}(\omega)}
\newcommand{\lam}{\lambda}
\newcommand{\Hom}{\text{Hom}}
\newcommand{\Ham}{{\rm{Ham}}}
\newcommand{\ham}{\text{ham}}
\newcommand{\Ker}{\text{Ker}}
\newcommand{\dist}{\text{dist}}
\newcommand{\psl}{\rm{PSL}}
\newcommand{\rk}{\roman{rk }}
\newcommand{\id}{\text{id}}
\newcommand{\psh}{\text{PSH}}
\newcommand{\Det}{\text{Det}\,}
\newcommand{\re}{\text{Re}\,}
\renewcommand\qed{ }
\begin{titlepage}
\title{\bf To the geometry of spaces of plurisubharmonic functions on a K\"ahler manifold}
\author{L\'aszl\'o Lempert \thanks{Research partially  supported by NSF grant DMS 1764167.
\newline 2020 Mathematics Subject classification 32Q15, 32U15, 53C35, 58B20, 58E30, 70H99}\\ 
%\newline 2020 Mathematics subject classification 53D05, 58D19}\\
Department of  Mathematics\\
Purdue University\\West Lafayette, IN
47907-2067, USA}
\thispagestyle{empty}
\end{titlepage}
\date{}
\maketitle
\abstract
Consider a compact K\"ahler manifold $(X,\omega)$ and the space $\cE(X,\omega)=\cE$ of $\omega$--plurisubharmonic functions of full Monge--Amp\`ere mass on it. We introduce a quantity $\rho[u,v]$ to measure the distance between 
$u, v\in\cE$; $\rho[u,v]$ is not a number but rather a decreasing function on a certain interval $(0,V)\subset\bR$. We explore properties of $\rho[u,v]$, and using them we study Lagrangians and associated energy spaces of
$\omega$--plurisubharmonic functions.  Many results here generalize Darvas's findings about 
his metrics $d_\chi$.
\endabstract

\section{Introduction}    % Section 1

Consider an $n$ dimensional connected compact K\"ahler manifold $(X,\omega)$ and the space $\text{PSH}(X,\omega)$ of $\omega$--plurisubharmonic functions $X\to [-\infty,\infty)$. (Its definition, as well as other background material will 
be reviewed in section 2.) Over thirty years now this space and its subspaces have been studied by endowing them with various 
metrics, and introducing special classes of paths, (weak) geodesics, in them. In this paper we focus on the space 
$\cE(X,\omega)=\cE\subset \psh(X,\omega)$ of functions of full Monge--Amp\`ere mass, and study a quantity $\rho[u,v]$, that we call the rise between $u,v\in\cE$, a notion that is related to the metrics, but is more fundamental. To define it we have to review the idea of geodesics in $\cE$ and some of their properties.

Given real numbers $a<b$, consider the strip
\[
S_{ab}=\{s\in \bC: a<\re s <b\}
\]
and let $\pi: S_{ab}\times X\to X$ be the projection. Following Berndtsson and Darvas \cite{Be1}, \cite[section 3.3]{Da3} we make the following definition.

\begin{defn}    %1.1
(a) A path (i.e., a map) $\varphi :(a,b)\to \psh(X,\omega)$ is a subgeodesic if $\Phi : S_{ab}\times X\to [-\infty,\infty)$ defined by $\Phi(s,x)=\varphi(\re s)(x)$ is $\pi^\ast\omega$--plurisubharmonic. 

(b) Given $u,v\in\psh (X,\omega)$, the geodesic $\psi :(a,b)\to \psh (X,\omega)$ joining them is 
\begin{equation}   %1.1
\psi=\sup\{\varphi \mid \varphi:(a,b)\to\psh (X,\omega)\,\text{is subgeodesic, } \lim_a\varphi \le u, \,\lim_b\varphi\le v\},
\end{equation}   %1.1
provided the supremum is not identically $-\infty$.
\end{defn}
In (1.1) the limits are understood pointwise on $X$; they exist because $\pi^\ast\omega$-plurisub-harmonicity of the associated function $\Phi$ implies that $\varphi(\cdot)(x)$ is convex (or $\equiv -\infty$) for all $x\in X$. We will only be interested in geodesics when $u, v\in \cE$. In this case the connecting geodesic $\psi$ is also a subgeodesic, maps into $\cE$ and 
$\lim_a \psi=u$, $\lim_b \psi=v$ in $L^1(X)$ and in capacity (and if $u,v$ are bounded, the limits are uniform) \cite[Corollaries 5.3, 5.4, and (23)]{Da2}, \cite[pp.156--157]{Be2}. Accordingly, it is natural to define $\psi(a)=u$, $\psi(b)=v$, and refer to the function $\psi:[a,b]\to \cE$ thus extended as a geodesic.

Traditionally, what we call geodesics here are termed weak (or plurisubharmonic or maximal) 
geodesics, and the name geodesic is reserved for $\psi$ in (1.1) that define a smooth path into the space $\cH$ of K\"ahler potentials,
\begin{equation}   %1.2
\cH=\{u\in C^\infty(X) : \omega+i\pa\op u >0\}\subset \psh(X,\omega).
\end{equation}
Such $\psi$ have velocity $\dot\psi:[a,b]\to T\cH$ that is parallel for a certain connection $\nabla$ on the tangent bundle $T\cH$ (Mabuchi's connection \cite{M}). However, since what tradition calls weak geodesics turned out to be rather more important than smooth 
solutions of the geodesic equation $\nabla_{\dot\psi} {\dot\psi}=0$, they earned the right to a short name, geodesic, and this is the name we will be using.  

Over the years geodesics in various subspaces of $\cE$ have been found to be subject to conservation laws 
\cite[Corollary 3.19]{S}, \cite[Proposition 2.2]{Be2}, \cite{Da1}, \cite[Theorem 4.4]{DNL}. The conserved 
quantities can be expressed in terms of the velocity of the geodesic and its decreasing rearrangement. As said, if $\psi:[a,b]\to\cE$ is a geodesic and $x\in X$, the function 
$\psi(\cdot)(x)$ is convex on $(a,b)$, and is either everywhere finite or $\equiv -\infty$. In the former case it has left and right derivatives, that we denote $\partial^-_t\psi(t)(x), \partial^+_t\psi(t)(x)$ or 
$\partial^{\pm}\psi$. We define the derivatives $\partial^{+}_t\psi(a)(x),\partial^{-}_t\psi(b)(x)\in[-\infty, \infty]$ 
as limits of $\partial^{\pm}_t\psi(t)(x)$ at $t=a,b$.
If $\psi(\cdot)(x)\equiv -\infty$, we define $\partial^\pm_t\psi(t)(x)=0$, but this convention will be of little importance. Both 
$\partial^\pm_t\psi(t): X\to [-\infty, \infty]$ are Borel functions; these are the left and right velocities of $\psi$. For $u\in\cE$ we denote by $\mu_u$ the Borel measure on $X$ induced by the non--pluripolar product $(\omega +i\partial\op u)^n$; the full mass condition means that 
$\mu_u(X)=\mu_0(X)=\int_X \omega^n$. We will abbreviate
\[
\mu_0(X)=V.
\]

If $a\le t<b$, resp. $a<t\le b$, we view $\partial^+_t \psi(t), \partial^-_t \psi(t)$ as functions on the measure space 
$(X,\mu_{\psi(t)})$, and form their decreasing rearrangements
\[
(\partial^+_t \psi(t))^\star, \ (\partial^-_t \psi(t))^\star : (0, V)\to \bR.
\]
It follows from \cite[Lemma 4.10]{Da1} that for geodesics $\psi$ that map into
\begin{equation}   %1.3
\cH^{1\bar 1}=\{u\in \psh(X,\omega): \text{the current } \partial\op u \text{ is represented by a bounded form}\},
\end{equation}
these rearrangements are conserved: $(\partial^-_t \psi(t))^\star$ and $(\partial^+_s \psi(s))^\star$ agree if $a<t\le b$, 
$a\le s<b$. (In Darvas's lemma $\dot u_t$ should be interpreted as right or left
derivative.) At the same time, Darvas 
\cite[section 7]{Da2} argues that for geodesics in the somewhat larger space of Lipschitz continuous functions, right and left derivatives at the endpoints $a, b$ may have different rearrangements; and \cite[Example 5.4]{L2} features a geodesic 
$\psi$ in the same space for which $(\partial^-_t \psi(t))^\star\ne (\partial^+_t \psi(t))^\star$ for all $t$ (see 
\cite[(5.5) and the following discussion]{L2}).

Nevertheless, it is possible to generalize the conservation law for geodesics in $\cH^{1\bar 1}$ to geodesics in $\cE$. This is based on the following: 
\begin{thm}[See Theorem 5.7]       %1.2
 Consider geodesics $\psi:[a,b]\to\cE$ and $\psi_j:[a,b]\to \cH^{1\bar 1}$. If $\psi_j(t)$ decreases to $\psi(t)$ for all $t\in[a,b]$, then there is a decreasing upper semicontinuous $(usc)$ function $\rho:(0,V)\to\bR$ to which the decreasing rearrangements $(\partial^-_t \psi_j(t))^\star$ and  $(\partial^+_t \psi_j(t))^\star$ converge Lebesgue almost everywhere as $j\to\infty$, for all $t\in (a,b)$. This $\rho$ depends only on $\psi$ and not on the choice of $\psi_j$.
\end{thm}

Note that for monotone functions a.e. convergence is the same as convergence at each point of continuity of the limit function, and the same as convergence on a dense set. Furthermore, any geodesic $\psi:[a,b]\to\cE$ can be obtained as the decreasing limit of geodesics $\psi_j:[a,b]\to \cH^{1\bar 1}$. This is so because one can choose $u_j$, $v_j\in\cH$ that decrease to $\psi(a)$, resp.  $\psi(b)$ (the customary references are \cite{De, DP} and especially \cite{BK}); the geodesics $\psi_j:[a,b]\to\cE$ joining $u_j,v_j$ map in fact into $\cH^{1\bar 1}$ by the work of Chen, with complements by B{\l}ocki \cite{Bl, C}; and $\psi_j$ will decrease to $\psi$ by \cite[Proposition 3.15]{Da3}.

\begin{defn}   % 1.3
We call $\rho=\rho_\psi:(0,V)\to\bR$ of Theorem 1.2 the rise of the geodesic $\psi$.
\end{defn}

At first sight, $\rho_\psi$ has nothing to do with conserved quantities; but it is conserved in the sense that for $[c,d]\subset [a,b]$, the rise of the geodesic $\psi |[c,d]$ is independent of the choice of $c,d$. This is obvious from the way $\rho$ is obtained in Theorem 1.2.

Once we have the notion of the rise of a geodesic, we can also talk about the rise $\rho[u,v]$ between $u,v\in\cE$: it is the rise of the geodesic $\psi:[0,1]\to \cE$ joining $u,v$,
\begin{equation}
\rho[u,v]=\rho_\psi.
\end{equation}
This is a distance--like quantity, except it measures distance between $u,v$ not by a number, but by 
a decreasing function $(0,V)\to \bR$.

The notion of the rise is useful because it contains metric information. The distance between $u,v$ measured in any of the 
Orlicz--type metrics $d_\chi$ that Darvas studies in [Da1--3], and the length of a geodesic measured in $d_\chi$, are easily 
recovered from the rise, see Definition 10.4 and (10.4). For example, 
$L^p$ distance $d_p(u,v)$ is just the $L^p$ norm of $\rho[u,v]$, and 
similarly for $d_\chi$ and the more general notion of action $\cL_T$ of \cite{L2}, associated with a Lagrangian. (The 
simplest instance of action is the increment  of
Monge--Amp\`ere energy between $u$ and $v$: it is $\int_0^V\rho[u,v]$, if $u,v$
belong to the energy space $\cE^1\subset\cE$, that we will introduce later.) At the same time, 
properties of the rise can be formulated independently of any choice of weight function $\chi$ or Lagrangian $L$. It is in this sense that the rise is more fundamental than the metrics.

The aim of this paper is to explore properties of the rise, of geodesics and between points in $\cE$; and to use them to
study Lagrangians on spaces of $\omega$--plurisubharmonic functions. The reader will notice that many of our results 
correspond to Darvas's results on $d_\chi$ in [Da 1--3]. Even if the results here appear to be more general, often our proofs can 
simply be extracted from Darvas's proofs. 

Contents. After reviewing background  material in section 2, we dicuss rise in $\cH^{1\bar 1}$ and its relation with envelopes of functions in sections 3, 4. Section 5 defines the rise in $\cE$ and formulates its main properties. In section 6 we show how, using decreasing rearrangements, one can compare velocities of geodesics with velocities of rectilinear paths. Section 7 is an admission that even if we can work with the notion of rise, its true meaning escapes us. The last three sections bring back Lagrangians to study metric properties of certain subspaces of $\cE$, show how metric notions can be directly reduced to the notion of rise, and formulate a Principle of Least Action, generalizing \cite[Theorem 8.1]{L2}.  

\section{Background}  %Section 2

If $Y$ is a complex manifold and $\Omega$ a real $(1, 1)$ form on it, $d\Omega=0$, a function $u:Y\to [-\infty, \infty)$ is said to be $\omega$--plurisubharmonic if for every open set $U\subset Y$ on which $\Omega$ can be written as $i\partial\op f$, the function $u+f$ is plurisubharmonic on $U$. We write $\psh(Y,\Omega)$ for the set of $\Omega$--plurisubharmonic functions. Most of the time we will be interested in $\omega$--plurisubharmonic functions on a connected compact K\"ahler manifold $(X,\omega)$ of dimension $n$. With any $u\in\psh(X,\omega)$ one can associate a Borel measure $\mu_u$ on $X$, induced by $(\omega+i\partial\op u)^n$, see [GZ1--2]. We let $V=\int_X\omega^n=\mu_0(X)$, and define the space $\cE\subset \psh(X,\omega)$ consisting of those $u\in\psh(X,\omega)$ for which
\begin{equation}   %2.1
\mu_u(X)=V.
\end{equation}
By $\cE^\infty$ we denote the space of bounded functions in $\psh(X,\omega)$; they are all contained in $\cE$. In (1.2), (1.3) we have already introduced the subspaces $\cH\subset \cH^{1\bar 1}\subset\cE^\infty$. He, Berndtsson, and Darvas proved that geodescics between points in $\cH^{1\bar 1}, \cE^\infty$, resp. $\cE$ stay entirely in those spaces \cite{H}, 
\cite[Section 2.2]{Be1}, \cite[Corollary 5.4]{Da2}. Not much is known about the regularity of these geodesics, but by work of Chen and B{\l}ocki, geodesics with endpoints in $\cH$ are $C^1$ as maps into $C(X)$, \cite{Bl, C}.

Next we turn to rearrangements. If $(X,\mu)$ and $(Y,\nu)$ are measure spaces, almost everywhere defined measurable functions 
$\xi:X\to [-\infty,\infty]$ and $\eta: Y\to [-\infty,\infty]$ are said to be equidistributed, or strict rearrangements of one another, if $\mu(\xi^{-1} B)=\nu(\eta^{-1}B)$ for every Borel set $B\subset [-\infty,\infty]$. If $\mu(X)=\nu(Y)<\infty$, this is equivalent to $\mu(\xi>t)=\nu(\eta> t)$ for all $t\in\bR$. Our notation for equidistribution will be
\begin{equation}    % 2.2
(\xi,\mu)\sim(\eta,\nu),\quad \rm{or }\quad \xi\sim \eta,
\end{equation}
if the measures are understood. The decreasing rearrangement of $\xi$ is the decreasing usc function 
$\xi^\star : (0, \mu(X))\to [-\infty,\infty]$ that is equidistributed with $\xi$, when $(0,\mu(X))$ is endowed with Lebesgue measure. Thus $\mu(s\le\xi\le t)$ is the length of the longest interval on which $s\le\xi^\star\le t$, $s,t\in [-\infty,\infty]$. If 
$\xi$ is $\mu$--almost everywhere finite, then $\xi^\star$ is finite everywhere. We have (see for example \cite[(3.2), (3.3)]{L2})
\begin{gather}   % 2.3
\mu\big(\xi>\xi^\star(s)\big)\le s\le\mu\big(\xi\ge\xi^\star(s)\big)\qquad \text{and}\\
 % 2.4
\mu(\xi\ge \tau)< s \, \text{ implies }\,  \xi^\star(s)<\tau,\qquad \mu(\xi>\tau)>s \,\text{ implies }\, \xi^{\star}(s)>\tau,
\end{gather}
because, e.g., the Lebesgue measure of $\big(\xi^\star >\xi^\star(s)\big)\subset (0, s)$ equals 
$\mu\big(\xi >\xi^\star (s)\big)$.

\begin{lem}    %2.1
Suppose $\mu(X)<\infty$ and a.e. defined measurable functions $\xi_j:X\to [-\infty,\infty]$ converge a.e. to $\xi:X\to\bR$. 
If $\xi^\star$ is continuous at some $s\in\big(0, \mu(X)\big)$, then $\lim_{j\to\infty} \xi^\star_j(s)=\xi^\star(s)$.
\end{lem}

\begin{proof}
We can assume $\xi_j\to\xi$ everywhere. If $G\subset [-\infty,\infty]$ is open, then 
$(\xi\in G)\subset\bigcup_{i\ge 1} \bigcap_{j\ge i}(\xi_j\in G)$. Therefore, given $\delta>0$, for sufficiently large $j$
\[
\mu(\xi_j\in G)>\mu(\xi\in G)-\delta \quad\text{and}\quad \mu(\xi_j\not\in G)<\mu(\xi\not\in G)+\delta.
\]
Let $0<s<\mu(X)$ and apply these estimates with $G=(\xi^\star(s+\delta)-\delta,\infty]$, resp. 
$G=[-\infty, \xi^\star(s-\delta)+\delta)$, in conjunction with (2.3) to obtain
\begin{align*}
&\mu\big(\xi_j>\xi^\star(s+\delta)-\delta\big)>\mu\big(\xi\ge \xi^\star(s+\delta)\big)-\delta\ge s,\\
&\mu\big(\xi_j\ge\xi^\star(s-\delta)+\delta\big)<\mu\big(\xi> \xi^\star(s-\delta)\big)+\delta\le s.
\end{align*}
(2.4) now implies $\xi^\star(s+\delta)-\delta<\xi^\star_j(s)<\xi^\star(s-\delta)+\delta$ for sufficiently large $j$. If $\xi^\star$ is continuous at $s$, then $\xi^\star_j(s)\to \xi^\star(s)$ follows by letting $\delta\to 0$.
\end{proof}

\begin{lem}  %2.2
Suppose $\mu(X)<\infty$ and $\xi,\eta:X\to\bR$ are a.e. defined measurable functions. If $F:\bR^2\to\bR$ is convex, and increasing
in both variables, then
\begin{equation} %2.5
F(\xi,\eta)^\star(s)\le F\big(\xi^\star(\sigma),\eta^\star(s-\sigma)\big),\qquad 0<\sigma<s<\mu(X). 
\end{equation}
\end{lem}
\begin{proof}
It suffices to prove (2.5) for a dense set of $(\sigma,s)$, because then to arbitrary $(\sigma,s)$ we can converge by $(\sigma_j,s_j)$
in this set in such a way that $\sigma_j<\sigma$, $s_j<s$, $s_j-\sigma_j<s-\sigma$; and (2.5) is obtained in the limit.
 
Consider first $F(x,y)=x+y$. If $\xi,\eta\ge 0$ are bounded, then $(\xi+\eta)^\star(s)\le \xi^\star(\sigma)+\eta^\star(s-\sigma)$ holds by
\cite[p. 41, (1.16)]{BS}. (Bennet and Sharpley use a different notion of rearrangement, but for nonnegative functions the two agree 
Lebesgue a.e.) The same also holds for arbitrary bounded $\xi,\eta$, because adding a constant will make them $\ge0$. 
Finally, (2.5) is obtained for
general $\xi,\eta$ by approximating them with bounded $\xi_j,\eta_j$, letting $j\to\infty$, and applying Lemma 2.1.

From this (2.5) follows when $F(x,y)=ax+by+c$, with $a,b\ge 0$. Indeed,
\[
F(\xi,\eta)^\star(s)=(a\xi+b\eta)^\star(s)+c\le(a\xi)^\star(\sigma)+(b\eta)^\star (s-\sigma)+c=F\big(\xi^\star(\sigma),\eta^\star(s-\sigma)\big).
\]
Since a general $F(x,y)$ can be represented as $\sup_{a,b,c} ax+by+c$ with a certain family of triples $a,b\ge0$, $c\in\bR$, (2.5) holds
in complete generality.
\end{proof}

We will talk about decreasing rearrangements of Borel functions $\xi$ defined on our K\"ahler manifold $(X,\omega)$, and the measure will be
 $\mu_u$ for some $u\in\cE$. If $\xi$ is naturally associated with a $u\in\cE$, for example because it arises as left or right velocity 
 $\partial^\mp_t\varphi(t)$ of a path $\varphi: [a,b]\to\cE$ and $u=\varphi(t)$, the measure will be $\mu_u$, and notation, as before
\begin{equation}  %2.6
\xi^\star=(\partial^{\mp}_t\varphi(t))^\star : (0, V)\to [-\infty, \infty], \quad \text{cf. } \  (2.1).
\end{equation}
Yet if $\xi$ is not clearly associated with some $u\in\cE$, we will have to indicate with respect to which measure $\mu_u$ we rearrange, and the notation will be 
\begin{equation}   %2.7
\xi^{\star u}:(0,V)\to [-\infty,\infty], \qquad u\in\cE.
\end{equation}

As said, \cite[Lemma 4.10]{Da1} implies:

\begin{lem}     %2.3
If $\varphi:[a,b]\to \cH^{1\bar 1}$ is a geodesic, then $(\partial^-_t\varphi(t))^\star=(\partial^+_s\varphi(s))^\star$ for 
$a< t\le b$, $a\le s<b$. 
\end{lem}

Among the subspaces of $\cE$ of interest, only $\cH$ is a manifold. Indeed, as an open subset of the Fr\'echet space $C^\infty(X)$, it inherits the structure of a Fr\'echet manifold; its tangent bundle $T\cH$ has a canonical trivialization 
$T\cH\approx\cH \times C^\infty(X)\to\cH$. As in \cite{L2}, we will consider certain functions $L: T\cH\to \bR$, Lagrangians, that give rise to geometric notions on $\cH$ such as metrics, and more generally, actions. The Lagrangians of interest extend to Banach 
bundles larger than $T\cH$. We write $B(X)$ for the Banach space of bounded Borel functions $X\to\bR$, endowed with sup norm, and $T^\infty \cE=\cE\times B(X)$. This is a set theoretical Banach bundle over $\cE$, into which $T\cH$ embeds via the trivialization $T\cH=\cH\times C^\infty(X)$.

\begin{defn}    %2.4
By an invariant convex Lagrangian we mean a function $L:T^\infty\cE\to\bR$ that is convex and continuous on the fibers $T_u^\infty\cE\approx B(X)$, $u\in\cE$, and is strict rearrangement invariant in the sense that $L(\xi)=L(\eta)$ if $\xi\in T_u^\infty\cE$, $\eta\in T_v^\infty\cE$ are equidistributed as functions on $(X,\mu_u), (X,\mu_v)$. We say $L$ is strongly continuous if the following holds: whenever 
$\xi_j\in T^\infty_u\cE$ are uniformly bounded and converge $\mu_u$--almost everywhere, then $L(\xi_j)$ is also convergent.
\end{defn}

For example, $L(\xi)=\int_X |\xi|^p \,d\mu_u$, $\xi\in T_u^\infty\cE$, defines an invariant, convex, strongly continuous Lagrangian, and so does its $p$'th  root, $1\le p<\infty$. But structurally the simplest Lagrangians, that generate all invariant, convex, strongly continuous Lagrangians  in a precise sense, are obtained as follows.

\begin{lem}      %2.5
If $f:(0,V)\to \bR$ is a decreasing integrable function, then
\begin{equation}   %2.8
L(\xi)=\int^V_0 \xi^{\star u}f, \quad \xi\in T^\infty_u\cE\approx B(X),
\end{equation}
defines an invariant, convex, strongly continuous Lagrangian on $T^\infty\cE$.
\end{lem}

In (2.8) the integral is against Lebesgue measure.
\begin{proof}
If $\xi\in T^\infty_u\cE$ and $\eta\in T^\infty_v\cE$ are equidistributed, then $\xi^{\star u}=\eta^{\star v}$ and $L(\xi)=L(\eta)$. 
For the rest of the proof we fix $u\in\cE$, consider functions $\xi\in T^\infty_u\cE\approx B(X,\mu_u)$, and will omit the $u$ from the notation $\xi^{\star u}$ of decreasing rearrangement.

To prove that $L$ is convex, we can assume $f$ in (2.8) is usc.  We claim
\begin{equation}     %2.9
L(\xi)=\sup \Big\{\int_X\gamma\xi\, d\mu_u \mid \gamma: X\to\bR\, \text{ is Borel},\, \gamma^\star=f\Big\}.
\end{equation}
Indeed, let $\theta:X\to (0,V)$ preserve measure, when $X$ is endowed with $\mu_u$ and $(0,V)$ with Lebesgue measure, restricted to the Borel sets. The existence of such $\theta$ follows e.g. from \cite[Proposition 7.4, p. 81]{BS}. Then 
$\gamma^\star=f$ is equivalent to $\gamma\sim f\circ\theta$. This shows that it suffices to prove (2.9) when $f$ is bounded, the general case will follow by approximation. Since adding a constant $c$ to $f$ changes both sides of (2.9) by $c\int_X\xi\, d\mu_u$, we can even assume $f\ge 0$. Similarly, we can assume $\xi\ge 0$. But then our claim
 \[
 L(\xi)=\int^V_0(f\circ \theta)^\star \xi^\star=\sup\Big\{ \int_X \gamma\xi\,d\mu_u: \gamma\sim f\circ \theta\Big\}
 \]
 is an instance of \cite[Theorem 2.6, p. 49]{BS}.  By (2.9) $L$ is the supremum of linear forms, and must be convex.
 
 Finally, strong continuity follows from Lemma 2.1 and dominated convergence.
\end{proof}

For the rest of the section we fix an invariant, convex, strongly continuous Lagrangian $L:T^\infty\cE\to\bR$.
The Principle of Least Action \cite[Theorem 8.1]{L2} then says:
\begin{thm}      %2.6
 If $\varphi:[a,b]\to\cE^\infty$ is piecewise $C^1$, as a map into the Banach space $B(X)$, and $\psi:[a,b]\to\cE^\infty$ is a geodesic of class $C^1$, connecting $\varphi(a)$ and $\varphi(b)$, then 
\[\int^b_a L(\partial_t \varphi(t))\,{dt}\ge \int^b_a L(\partial_t\psi(t))\,dt.\]
\end{thm}

Given $T\in (0,\infty)$, the (least) action between $u,v\in \cE^\infty$ is 
\begin{equation}         %2.10
\cL_T(u,v)=\inf\int^T_0 L(\partial_t\varphi(t))\,dt,
\end{equation}
the infimum taken over all piecewise $C^1$ paths $\varphi:[0,T]\to\cE^\infty$ joining $u,v$. If the geodesic $\psi:[0,T]\to\cE^\infty$ joining $u,v$ is $C^1$, Theorem 2.6 implies.
\begin{equation}        %2.11
\cL_T(u,v)=\int^T_0 L(\partial_t\psi(t))\,dt.
\end{equation}
We will need the following continuity properties of $\cL_T, L$:
\begin{lem}      %2.7
If $u_j, v_j\in C(X)\cap \psh(X,\omega)$ decrease, or converge uniformly, to $u\in\cE^\infty$, resp. $v\in\cE^\infty$, then 
\[\lim_{j\to\infty} \cL_T(u_j,v_j)=\cL_T(u,v)\in \bR.\]
\end{lem}
This is \cite[Lemma 9.4]{L2}. By (the proof of) Lemma 10.5 below and the discussion around it, it  also holds if 
$u_j,v_j\in \cE^\infty$, and even more generally.

Let $B(0, V)$ stand for the Banach space of bounded Borel functions on $(0, V)$, with the sup norm. Any strict rearrangement invariant function $L: T^\infty\cE\to \bR$ determines (and is determined by) a function $L_\star: B(0,V)\to\bR$, 
\begin{equation}         %2.12
L_\star(\zeta)=L(\xi) \qquad \text{if } \zeta\in B(0, V)\, \text{ and } \xi\in T^\infty_u\cE\ \text{ are equidistributed.}
\end{equation}
In particular, $L(\xi)=L_\star(\xi^\star)$. Indeed, if $u\in\cE$ and
$\theta:(X, \mu_u)\to \big((0,V), \rm{Lebesgue}\big)$\footnote{Here and in what follows, we will always 
restrict Lebesgue measure on $(0,V)$ to the Borel $\sigma$--algebra.}
preserves measure, we can define $L_\star(\zeta)=L(\zeta\circ\theta)$.

\begin{lem}    %2.8
If uniformly bounded $\zeta_j\in B(0, V)$ converge a.e. to $\zeta\in B(0, V)$, then 
$L_\star(\zeta_j)\to L_\star(\zeta)$.
\end{lem}
\begin{proof}With a measure preserving $\theta:(X,\mu_u)\to (0,V)$ as before,
\[\lim_j L_\star(\zeta_j)=\lim_j L(\zeta_j\circ\theta)=L(\zeta\circ\theta)=L_\star(\zeta).\]
\end{proof}

\section{Rise in $\cH^{1\bar 1}$}       %Section 3

In this section we introduce the notion of rise in $\cH^{1\bar 1}$. 

\begin{defn}      % 3.1
The rise $\rho_\varphi:(0, V)\to\bR$ of a geodesic $\varphi:[a,b]\to \cH^{1\bar 1}$ is 
$(\partial^-_t\varphi(t))^\star=(\partial^+_s\varphi(s))^\star$ for any $a<t\le b$, $a\le s<b$, cf. Lemma 2.3. 
The rise between $u,v\in \cH^{1\bar 1}$, denoted $\rho[u,v]$, is the rise $\rho_\varphi$ of the geodesic 
$\varphi:[0, 1]\to\cH^{1\bar 1}$ joining $u$ and $v$.
\end{defn}

An affine reparametrization  $\varphi(pt+q)$ of a geodesic $\varphi$, where $p,q\in\bR$, is also a geodesic, because the same holds for subgeodesics. This implies that with an arbitrary geodesic $\varphi:[a,b]\to\cH^{1\bar 1}$
\begin{equation}      %3.1
\rho[\varphi(a),\varphi(b)]=(b-a)\rho_\varphi.
\end{equation}

A geodesic $\varphi:[a,b]\to \cH^{1\bar 1}$ can be reversed to produce a geodesic $\psi:[-b,-a]\to\cH^{1\bar 1}$, 
$\psi(t)=\varphi(-t)$. It follows that the function $-\rho_\varphi(V-\cdot)$ is a decreasing function, equidistributed with 
$\rho_\psi$. In general, it will not be usc, so all we can say is that
\begin{equation}      %3.2
\rho_\psi(\lambda)=-\rho_\varphi(V-\lambda) \quad \text{for a.e. } \lambda\in(0,V).
\end{equation}
\begin{lem} %3.2
Suppose $u,v,w\in\cH^{1\bar1}$ and $c\in\bR$.
\begin{itemize}
\item[(a)] If $v\le w$ then $\rho[u,v]\le \rho[u,w]$ and $\rho[v,u]\ge\rho[w,u]$.
\item[(b)] If $c\in\bR$ then $\rho[u, v+c]=\rho[u,v]+c=\rho[u-c, v]$.
\end{itemize}\end{lem}
\begin{proof}
Let $\varphi_c:[0,1]\to\cH^{1\bar 1}$ be the geodesic that joins $u$ with $v+c$, and $\psi:[0,1]\to\cH^{1\bar 1}$ the geodesic that joins $u$ with $w$.
Definition 1.1 implies that $\varphi_0(t)\le\psi(t)$, and
$\varphi_c(t)=\varphi(t)+ct$. Since $\varphi_0(0)=\psi(0)$,
\[
\partial^+\varphi_0(0)\le\partial^+\psi(0) \quad\text{and}\quad \partial_t^\pm\varphi_c(t)=\partial_t^\pm\varphi_0(t)+c.
\]
Hence the corresponding decreasing rearrangements are related in the same way, $\rho[u,v]\le\rho[u,w]$ and $\rho[u,v+c]=\rho[u,v]+c$. The rest of the lemma can be proved similarly.
\end{proof}

\begin{lem}     % 3.3
Consider geodesics $\varphi, \varphi_j: [a,b]\to \cH^{1\bar 1}$, $j\in \bN$, and assume that the $(1,1)$ forms 
$\partial\op\varphi_j(a), \pa\op\varphi_j(b)$ on $X$ are uniformly bounded. If $\varphi_j(a), \varphi_j(b)$ decrease, or converge uniformly, to $\varphi(a)$, resp. $\varphi(b)$, then $\rho_{\varphi_j}\to\rho_\varphi$ a.e. More precisely, the latter convergence holds at all points of continuity of $\rho_\varphi$.

Equivalently, if $u_j, v_j\in\cH^{1\bar 1}$ decrease or converge uniformly to $u,v$, and
$\partial\op u_j,\partial\op v_j$ are uniformly bounded, then $\rho[u_j,v_j]\to\rho[u,v]$ at all points of continuity of $\rho[u,v]$.
\end{lem}

\begin{proof}
Consider the case when $\varphi_j(a), \varphi_j(b)$ decrease. As $j\to\infty$, by \cite[Proposition 3.15]{Da3} 
$\varphi_j(t)$ decrease to $\varphi(t)$ for every $t\in[a,b]$. Further, as explained in the proof of \cite[Lemma 11.2]{L2}, 
it follows from He's work \cite{H} that the family $\mu_{\varphi_j(t)}$, $\mu_{\varphi(t)}$, $j\in\bN$, $t\in[a,b]$, of measures is 
hereditarily tight, in the sense that given an open $U\subset X$ and $\varepsilon>0$, there exists a compact $K\subset U$ such 
that 
$\mu_{\varphi_j(t)}(U\setminus K)$, $\mu_{\varphi(t)}(U\setminus K)<\varepsilon$. 

Fix $t\in(a,b)$. By convexity $\partial^-\varphi(t)\le\partial^+\varphi(t)$ on $X_j$. Since the two functions are equidistributed, it follows that $\partial^-\varphi(t)=\partial^+\varphi(t)$ $\mu_{\varphi(t)}$--almost everywhere. 
In other words, for $\mu_{\varphi(t)}$--almost every $x\in X$ the function $\varphi(\cdot)(x)$ is differentiable at 
$t$. For brevity, we denote the derivative $\dot\varphi(t)(x)$. It follows that $\partial^+\varphi_j(t)(x)\to \dot\varphi(t)(x)$ at such $x$. Indeed, given $\varepsilon>0$, choose $s>0$ so that
\begin{equation}      %3.3
\Big|\frac{\varphi(t\pm s)(x)-\varphi(t)(x)}{\pm s}-\dot\varphi(t)(x)\Big|<\varepsilon.
\end{equation}
By convexity
\[
\frac{\varphi_j(t- s)(x)-\varphi_j(t)(x)}{- s}\le\partial^+ \varphi_j(t)(x)\le \frac{\varphi_j(t+ s)(x)-\varphi_j(t)(x)}{s}.
\]
As $j\to\infty$, the upper and lower bounds of $\partial^+ \varphi_j(t)(x)$ tend to $\big(\varphi(t\pm s)(x)-\varphi(t)(x)\big)/\pm s$. 
Hence $\partial^+ \varphi_j(t)(x)\to\dot\varphi(t)(x)$ by (3.3), i.e., $\partial^+ \varphi_j(t)\to\dot\varphi(t)$ $\mu_{\varphi(t)}$--a.e. on $X$..

This and hereditary tightness imply by \cite[Lemma 3.6]{L2}
\[
\rho_{\varphi_j}=(\partial^+ \varphi_j(t))^{\star \varphi_j(t)}\to \dot\varphi(t)^{\star \varphi(t)}=\rho_\varphi,\qquad j\to\infty,
\]
away from a countable subset of $(0, V)$. As the functions involved are monotone, the exceptional set will be included in the set of discontinuity of $\rho_\varphi$, as claimed.

The case of uniformly convergent $\varphi_j(a), \varphi_j(b)$ can be reduced in a standard way to the decreasing case, see for example the last paragraph of the proof of \cite[Lemma 3.4]{L2}.
\end{proof}

\begin{lem}          %3.4
Suppose $L:T^\infty\cE\to\bR$ is an invariant, convex, strongly continuous Lagrangian. If $T\in (0,\infty)$ and $u,v\in\cH^{1\bar 1}$, then (cf. (2.10), (2.12))
\[\cL_T(u,v)=TL_\star (\rho[u,v]/T).\]
\end{lem}
\begin{proof}
First assume $u,v\in\cH$, so that the geodesic $\psi:[0,T]\to\cH^{1\bar 1}$ between them is $C^1$, viewed as a map into $B(X)$. By (2.11), (2.12), and (3.1)
\[
\cL_T(u,v)=\int^T_0L(\dot\psi(t))\,dt=\int^T_0 L_\star(\rho_\psi)\,dt=TL_\star(\rho[u,v]/T).
\]

For general $u,v\in \cH^{1\bar 1}$ it is easy to find sequences $u_j, v_j\in\cH$ with $\partial\op u_j,\partial\op v_j$ uniformly
bounded, that converge uniformly to $u,v$. (One
approximates, say, $(1-1/j)u$ by convolutions on coordinate neighborhoods, then patches together the local approximants 
by a smooth partition of unity.)
By Lemmas 2.7, 3.2, 3.3, and 2.8, and by dominated convergence
\[\cL_T(u,v)=\lim_j\cL_T(u_j, v_j)=\lim_j TL_\star(\rho[u_j, v_j]/T)=TL_\star(\rho[u,v]/T).\]
\end{proof}

\section{Envelopes and rise in $\cH^{1\bar 1}$}    %Section 4

Various spaces of $\omega$--plurisubharmonic functions on $(X,\omega)$ form lattices. We will use $\vee,\wedge$ to denote the lattice operations, envelopes. In general, if $u,v\in\psh(X,\omega)$, we define $u\vee v, u\wedge v: X\to[-\infty,\infty)$ by
\begin{equation}\begin{gathered}       %4.1
(u\vee v)(x)=\max(u(x), v(x))\, \quad\text{ and }\\
(u\wedge v)(x)=\sup\{w(x): w\in\psh(X,\omega), \ w\le u, v\}.
\end{gathered} \end{equation}   %4.1
Thus $u\vee v\in\psh(X,\omega)$ and, according to  \cite[Section 2.4]{Da3}, $u\wedge v\in\psh(X,\omega)$ or $\equiv -\infty$. 
Darvas proves that if $u,v\in\cE$, then $u\vee v, u\wedge v\in\cE$ \cite[Corollaries 2.7, 3.5]{Da2}. 
Also, if 
$u,v\in\cH^{1\bar 1}$, then $u\wedge v\in\cH^{1\bar 1}$ by \cite[Theorem 2.5]{DR}. Darvas and Rubinstein write $P(u,v)$ for what is 
denoted $u\wedge v$ here. That $\wedge$ is used both for envelope of functions and for exterior product of forms might 
confuse the careless reader; but since these two occur in different contexts, and $u\wedge v$ is rather simpler than $P(u,v)$, 
adopting it is worth the risk.

\begin{lem}      %4.1
Let $u,v\in\cH^{1\bar 1}$, and $\varphi,\psi,\theta:[a,b]\to \cH^{1\bar 1}$ the geodesics joining $u$ with $u\wedge v$, $u\wedge v$ with $v$, and $u$ with $v$. Then
\begin{equation}      %4.2
\partial^+\varphi(a)=\min \big(0, \partial^+\theta(a)\big)\quad \text{and}\quad  \partial^-\psi(b)=\max \big(0, \partial^-\theta(b)\big).
\end{equation}
Hence $\rho_\varphi=\min(0,\rho_\theta)$,  $\rho_\psi=\max(0,\rho_\theta)$, and
\[
\rho[u,u\wedge v]+\rho[u\wedge v, v]=\rho[u,v].
\]
\end{lem}

The lemma corresponds to Darvas's Pythagorean Theorem, \cite[Proposition 4.13]{Da1}, \cite[Proposition 8.1]{Da2}.
\begin{proof}
We show that two functions are equal by checking that their level sets coincide. By \cite[Lemma 3.17]{Da3}, for any $\tau\in \bR$
\begin{equation}     %4.3
\big(\partial^+\theta(a)\ge\tau\big)=\big(u\wedge(v-\tau)=u\big),
\end{equation}
and replacing $v$ by $u\wedge v$
\[
\big(\partial^+\varphi(a)\ge\tau\big)=\big(u\wedge(u\wedge v-\tau)=u\big)=\big(u\wedge(u-\tau)\wedge(v-\tau)=u\big).
\]
This latter set is $\big(u\wedge(v-\tau)=u\big)$ if $\tau\le 0$, and empty otherwise.
Comparison with (4.3) gives the first formula in (4.2). The second follows by applying what we proved but with the roles of $u,v$ interchanged. The rest of the claim is immediate from (4.2).
\end{proof}
Next we look at how enveloping interacts with rise. The result corresponds to \cite[Proposition 8.2]{Da2}.

\begin{lem}         %4.2
Let $u,v,w\in\cH^{1\bar 1}$. If $u\le v$ then $\rho[u\wedge w, v\wedge w]\le\rho[u,v]$; if $u\ge v$ then
$\rho[u\wedge w, v\wedge w]\ge\rho[u,v]$.
\end{lem}
\begin{proof}
Assume  $u\le v$, and let $\varphi,\psi:[0, 1]\to\cH^{1\bar 1}$ be the geodesics between $u=u\wedge v$ and $v$, respectively $u\wedge w$ and $v\wedge w$. Lemma 4.1 implies $\rho_\varphi\ge 0$, and similarly $\rho_\psi\ge 0$. Using again 
\cite[Lemma 3.17]{Da3}, for $\tau\in \bR$
\begin{gather}    %4.4,5
\mu_u\big(\partial^+\varphi(0)\ge\tau\big)=\mu_u\big(u\wedge(v-\tau)=u\big),\\ 
\mu_{u\wedge w}\big(\partial^+\psi(0)\ge\tau\big)=\mu_{u\wedge w}\big(u\wedge w\wedge(v\wedge w-\tau)=u\wedge w\big).
\end{gather}
If $\tau\le 0$, both measures equal $V$. If $\tau>0$ then 
$w\wedge(v\wedge w-\tau)=w\wedge(v-\tau)\wedge(w-\tau)=(v-\tau)\wedge(w-\tau)$, and so the set on the right of (4.5) is 
$\big(u\wedge(v\wedge w-\tau)=u\wedge w\big)$. According to Darvas 
%(when $u,v\in\cH^{1\bar1}$) and Guedj, Lu, Zeriahi (when $u,v\in\cE$) 
\cite[Proposition 2.2]{Da2}, for Borel sets 
$Y\subset X$
\begin{equation}  %4.6
\mu_{u\wedge w}(Y)=\mu_u\big(Y\cap(u=u\wedge w)\big)+\mu_w\big(Y\cap(u>u\wedge w=w)\big).
\end{equation}
In light of (4.5), then
\begin{multline}    %4.7
\mu_{u\wedge w}\big(\partial^+ \psi(0)\ge\tau\big) \le
\mu_u\big(u\wedge(v\wedge w-\tau) =u\wedge w=u\big)\\
+\mu_w\big(u\wedge(v\wedge w-\tau)=u\wedge w=w\big).
\end{multline}       
Since $v\wedge w-\tau < w$, the second term on the right is 0. Furthermore
\[u\wedge(v\wedge w-\tau)\le u\wedge(v-\tau)\le u,\]
whence the first term on the right of (4.7) is $\le\mu_u(u\wedge(v-\tau)=u)$. Comparing (4.4) and (4.7) we 
therefore obtain 
$\mu_{u\wedge w}(\partial^+\psi (0)\ge \tau)\le \mu_u(\partial^+\varphi(0)\ge \tau)$, or
 \[\rho[u\wedge w, v\wedge w]=(\partial^+\psi(0))^\star\le (\partial^+\varphi(0))^\star=\rho[u,v].\] 
 A similar argument gives the lemma when $u\ge v$. 
 \end{proof}

Both lemmas support the view of $\rho[u, v]$ as measuring some distance between $u,v$. The question 
arises if $\rho$ satisfies the triangle inequality. It is easy to see that $\rho(u,v)+\rho(v,w)\ge\rho(u,w)$ 
will not hold for a general triple $u,v,w\in\cH$. Indeed, if it did, it would also hold for the triple $w,v,u$. But in 
light of (3.2), the two triangle inequalities would simply imply $\rho(u,v)+\rho(v,w)=\rho(u,w)$, something that fails if $u=w$ but $\rho(u,v)$ is not a constant.

However, the triangle inequality does hold in an integrated form:

\begin{lem}       %4.3
If $u,v,w\in\cH^{1\bar 1}$ and $\lambda\in[0,V]$, then 
\begin{equation}     %4.8
\int^\lambda_0\rho[u,v]+\int^\lambda_0\rho[v,w]\ge \int^\lambda_0\rho[u,w].
\end{equation}
\end{lem}

\begin{proof}
The proof will exhibit (4.8) as an instance of the Principle of Least Action, Theorem 2.6. In view of Lemma 3.2 it suffices to prove 
when $u,v,w\in\cH$. Let $\varphi:[0,2]\to\cH^{1\bar 1}$ be such that $\varphi|[0,1]$ is the geodesic connecting $u,v$ and $\varphi |[1,2]$ is the geodesic connecting $v,w$. Let $\psi:[0,2]\to\cH^{1\bar 1}$ be the geodesic connecting $u,w$. For Lagrangians $L$ as in Theorem 2.6 we have 
$\int^2_0 L(\partial_t\varphi(t))\,dt\ge \int^2_0 L(\partial_t\psi(t))\,dt$, or
\begin{equation}     %4.9
L_\star(\rho[u,v])+L_\star(\rho[v,w])\ge 2L_\star(\rho[u,w]/2)
\end{equation}
by (2.11) and Lemma 3.4. We apply this with the Lagrangian
\[
L(\xi)=\int^\lambda_0\xi^\star=\int^V_0 f\xi^\star, \quad \xi\in T^\infty\cE,
\]
where $f=1_{(0,\lambda)}$ is the characteristic function of $(0,\lambda)$, cf. Lemma 2.5. Since $L_\star(\xi^\star)=\int^\lambda_0\xi^\star$, (4.9) reduces to (4.8).
\end{proof}

If we introduce a partial order $\succeq$ on decreasing integrable functions on $(0,V)$ whereby
$f\succeq g$ means $\int_0^\lambda f\ge\int_0^\lambda g$ for $0<\lambda<V$ (known in harmonic analysis
as the Hardy--Littlewood--P\'olya relation), then (4.8) becomes $\rho[u,v]+\rho[v,w]\succeq\rho[u,w]$.

\section{The rise in $\cE$}     %section 5
We are ready to extend the notion of rise to geodesics $\varphi$ and functions $u,v$ in $\cE$.

\begin{thm}      %5.1
If sequences $u_j$, $v_j\in\cH^{1\bar 1}$ decrease to $u,v\in\cE$, then the functions $\rho[u_j, v_j]: (0, V)\to\bR$ converge a.e. to a decreasing usc function $r:(0,V)\to\bR$. The function $r$ depends on $u,v$ but not on the approximating sequences $u_j, v_j$. 
\end{thm}

\begin{defn}      %5.2
We call the function $r$ above the rise between $u,v$, and denote it by $\rho[u,v]$. If $\psi:[a,b]\to\cE$ is a geodesic, we define its rise by
\[
\rho_\psi=\frac{1}{b-a}\rho[\psi(a),\psi(b)].
\]
\end{defn}
Thus both types of rise are decreasing usc function $(0, V)\to\bR$. Since if $u,v\in\cH^{1\bar 1}$, 
we can choose $u_j=u, v_j=v$, within $\cH^{1\bar 1}$ the new notion of the rise agrees with the notion in Definition 3.1.

For the proof of Theorem 5.1 we have to recall some of Darvas's results in \cite{Da4}. With a concave smooth diffeomorphism 
$\chi:[0,\infty)\to[0,\infty)$ Guedj and Zeriahi associate the space
\[
\cE_\chi=\Big\{u\in\cE: \int_X\chi(|u|)\,d\mu_u <\infty\Big\}
\]
and Darvas defines a metric $d_\chi$ on $\cH^{1\bar 1}$ (which he subsequently extends to all of $\cE_\chi$). When $u,v\in\cH^{1\bar 1}$ and $\varphi:[0, 1]\to\cH^{1\bar 1}$ is the geodesic to join them,
\[
d_\chi(u,v)=\int_X\chi\big(|\partial^\pm_t\varphi(t)|\big)\,d\mu_{\varphi(t)},\qquad 0<t<1,
\]
see \cite[(24)]{Da4}. He proves \cite[Lemma 5.2.]{Da4}:
\begin {lem}       %5.3
If $u_j,v_j\in\cH^{1\bar 1}$ decrease to $u,v\in\cE_\chi$, then $d_\chi(u_j, v_j)$ is convergent and its limit depends only on $u,v$, not on the choice of $u_j, v_j$. This limit is denoted $d_\chi(u,v)$. 
\end{lem}
Note that in $\cH^{1\bar 1}$, $d_\chi$ can be expressed through $\rho[u,v]$, 
\begin{equation}      %5.1
d_\chi(u,v)=\int^V_0\chi\big(|(\partial^+\varphi(0))^\star |\big)=\int^V_0\chi(|\rho[u,v]|), \quad u,v\in\cH^{1\bar 1}.
\end{equation}
\begin{lem}      %5.4
Suppose $u_j,v_j\in\cH^{1\bar 1}$ decrease to $u,v\in\cE_\chi$, $u_j\le v_j$, and $\rho[u_j, v_j]$ converge a.e. to a function $r:(0, V)\to [0,\infty)$. Then
\[
d_\chi(u,v)=\lim_{j\to\infty} d_\chi(u_j,v_j)=\int^V_0\chi\circ r.
\]
\end{lem}

\begin{proof}
Since $\rho[u_j,v_j]\ge 0$, in light of (5.1) we need to show
\begin{equation}      %5.2
\lim_{j\to\infty}\int^V_0\chi(\rho[u_j, v_j])=\int^V_0\chi\circ r.
\end{equation}
As in the proof of \cite[Proposition 10.16]{GZ2}, one can construct a smooth increasing function $f:[0,\infty)\to [1,\infty)$ such that $\lim_\infty f=\infty$, $\ochi=f\chi:[0,\infty)\to[0,\infty)$ is still a concave diffeomorphism, and $u,v\in\cE_{\ochi}$. Thus the sequence $d_{\ochi}(u_j, v_j)=\int_0^V\ochi(\rho[u_j,v_j])$ is still convergent by Lemma 5.3; let $C=\sup_j d_{\ochi}(u_j, v_j)$.

With a positive number $M$ write
\[\rho_M[u_j,v_j]=\min(M,\rho[u_j, v_j])\quad \text{and}\quad r_M=\min(M, r).\]
By dominated resp. monotone convergence
\begin{equation}       %5.3
\lim_{j\to\infty}\int^V_0\chi(\rho_M[u_j, v_j])=\int^V_0\chi\circ r_M, \qquad \lim_{M\to\infty}\int^V_0\chi\circ r_M=\int^V_0\chi\circ r.
\end{equation}
Furthermore, using also (5.1)
\begin{multline*}
0\le\int^V_0\chi(\rho[u_j, v_j])-\int^V_0\chi(\rho_M[u_j, v_j])\le\int_{\rho[u_j, v_j]>M} \chi(\rho[u_j,v_j])\\
\le \frac{1}{f(M)}\int_{\rho[u_j, v_j]> M}\ochi(\rho[u_j, v_j])\le\frac{1}{f(M)}\int^V_0\ochi(\rho[u_j, v_j])\le\frac{C}{f(M)}.
\end{multline*}
Since $\lim_{M\to\infty} f(M)=\infty$, this estimate together with (5.3) gives (5.2).
\end{proof}

\begin{proof}[Proof of Theorem 5.1] By \cite[Proposition 10.16]{GZ2} there is a concave diffeomorphism 
$\chi:[0,\infty)\to[0,\infty)$ such that $u\in \cE_\chi$; and inspecting the proof we see that we can arrange 
$v\in\cE_\chi$ as well. By Lemma 5.3 the sequence $d_\chi(u_j, v_j)$ is convergent, hence bounded. Assume 
first that $u_j\le v_j$ for all $j$.
If $0<\lambda<V$, by (5.1)
\[
0\le\chi\big(\rho[u_j,v_j](\lambda)\big)\le\frac{1}{\lambda}\int^\lambda_0\chi(\rho[u_j,v_j])\le\frac{1}{\lambda} d_\chi(u_j, v_j).
\]
In particular, the sequences $\chi(\rho[u_j,v_j])$ and $\rho[u_j,v_j]$ are pointwise bounded. By Helly's theorem a subsequence
$\rho[u_{j_k},v_{j_k}]$ will converge everywhere to a decreasing function. By modifying the limit function at its points of discontinuity we obtain a decreasing usc function $r:(0, V)\to[0,\infty)$ such that $\lim_{k\to\infty}\rho[u_{j_k},v_{j_k}]=r$ at points where $r$ is continuous. Fix a $\tau\in(0, V)$ where $r$ is continuous. We claim that if $\ou_j, \ov_j\in \cH^{1\bar 1}$ decrease to $u,v$, and $\ou_j\le \ov_j$, then
\begin{equation}     %5.4
\rho[\ou_j,\ov_j](\tau)\to r(\tau);
\end{equation}
this will then prove the theorem whenever $u_j\le v_j$.

Indeed, suppose (5.4) fails, and $|\rho[\ou_j,\ov_j](\tau)-r(\tau)|\ge\delta$ with some $\delta >0$ and infinitely many $j$. Using Helly's theorem as before, we would then have a sequence $i_1< i_2<\dots$ and a decreasing usc function $\overline r:(0,V)\to [0,\infty)$ such that
\begin{equation}      %(5.5)
\lim_{k\to\infty}\rho[\ou_{i_k}, \ov_{i_k}]=
\overline r \quad\text{wherever $\overline r$ is continuous; but } \overline r(\tau)\ne r(\tau).
\end{equation}
By Lemma 5.4 $\int^V_0\chi\circ r=d_\chi(u,v)=\int^V_0\chi\circ \overline r$. But, if $f\in C^\infty[0,\infty)$ has compact support contained in $(0,\infty)$ and $\varepsilon >0$ is sufficiently small, then $\chi+\varepsilon f:[0,\infty)\to [0,\infty)$ is also a concave diffeomorphism, and $u,v\in\cE_{\chi+\varepsilon f}$. Therefore 
\[
\int^V_0(\chi+\varepsilon f)\circ r=\int^V_0(\chi+\varepsilon f)\circ \overline r,\qquad i.e.,\quad
 \int^V_0 f\circ r=\int^V_0f\circ \overline r.
\]
The latter even holds for characteristic functions $f=1_{[p,q]}$,  
where $0<p<q$, since these characteristic functions are decreasing limits of smooth functions $f$ with compact support in $(0,\infty)$. Therefore $r$ and $\overline r$ are equidistributed. As both decrease and are usc, $r=\overline r$ follows, in contradiction with (5.5). This proves the theorem when $u_j\le v_j$. Clearly, the theorem also holds if $u_j\ge v_j$, cf. (3.2).

Now consider general $u_j, v_j$. By \cite[Corollary 3.5]{Da2} and \cite[Theorem 2.5]{DR} $u\wedge v\in\cE$ and 
$u_j\wedge v_j\in\cH^{1\bar 1}$. A moment's thought gives that the $\omega$--plurisubharmonic function 
$w=\lim_j u_j\wedge v_j$ is equal to $u\wedge v$. Indeed, $u_j\wedge v_j\ge u\wedge v$ implies $w\ge u\wedge v$. Also $w\le u_j, v_j$, whence $w\le u,v$ and $w\le u\wedge v$. From Lemma 4.1
\begin{equation}       %5.6
\rho[u_j, u_j\wedge v_j]=\min(0, \rho[u_j, v_j]), \quad \rho[u_j\wedge v_j,v_j]=\max (0, \rho[u_j, v_j]),
\end{equation}
and we conclude that $\rho[u_j, v_j]=\rho[u_j, u_j\wedge v_j]+\rho[u_j\wedge v_j, v_j]$ indeed converges a.e.; the limit depends only on $u, u\wedge v$ and $v$, i.e., on $u,v$.
\end{proof}

\begin{thm}      %5.5
If $u,v\in\cE$, then
\[\rho[u, u\wedge v]=\min(0, \rho[u,v]),\quad \rho[u\wedge v, v]=\max(0,\rho[u,v]).\]
In particular, $\rho[u,v]=\rho[u, u\wedge v]+\rho[u\wedge v, v]$.
\end{thm}

This is immediate from (5.6) and Definition 5.2.---Theorem 5.1 has the following generalization: 

\begin{thm}    %5.6
If $u_j, v_j\in\cE$ decrease to $u,v$, then $\rho[u_j, v_j]\to\rho[u, v]$ at all points in $(0, V)$ where $\rho[u,v]$ is continuous.
\end{thm}

\begin{proof}
It will suffice to prove convergence along a subsequence. For each $j\in\bN$ let $u_{ij}, v_{ij}\in\cH, i=1, 2,\dots$, strictly 
decrease to $u_j, v_j$ (i.e., $u_{i+1,j}< u_{ij}$ and $v_{i+1,j} <v_{ij}$ everywhere). Similarly, choose $U^i, V^i\in\cH$ strictly decreasing to $u,v$. Let $\tau_1, \tau_2,\ldots\in (0, V)$ be a dense sequence of points where each $\rho[u,v], \rho[u_j, v_j]$ is continuous. By Theorem 5.1/Definition 5.2
\begin{equation}       %5.7
\rho[u_{ij}, v_{ij}]\to\rho[u_j, v_j]\qquad\text{at each } \tau_k\,\text{ as } i\to\infty.
\end{equation}

We define natural numbers $1=i(1)<i(2)<\dots$ and $1=j(1)<j(2)<\dots$ recursively, and set $U_k=u_{i(k)j(k)}$, 
$V_k=v_{i(k)j(k)}$, 
 as follows. 
Suppose we already have $i(k-1), j(k-1)$. If $j$ is sufficiently large, then $u_j< U_{k-1}, U^k$ and
$v_j<V_{k-1}, V^k$ by Dini's theorem. Fix such $j=j(k)>j(k-1)$. Since $\lim_{i\to\infty} u_{ij}=u_j<U_{k-1}, U^k$ 
and 
$\lim_{i\to\infty} v_{ij}=v_j<V_{k-1}, V^k$, for sufficiently large $i=i(k)> i(k-1)$ again by Dini's theorem and by (5.7)
\begin{equation} %5.8
u_{ij}< U_{k-1},  U^k,\quad v_{ij}< V_{k-1}, V^k, \quad\text{and}\quad 
|\rho[u_{ij}, v_{ij}]-\rho[u_j, v_j]|< 1/k 
\end{equation}
at $\tau_1,\dots, \tau_k$. Thus $U_k=u_{i(k)j(k)}$ decrease to $u$ and $V_k=v_{i(k)j(k)}$ decrease to $v$.
By Theorem 5.1/Definition 5.2 therefore $\rho[U_k, V_k]\to\rho[u,v]$ a.e.; and in fact by monotonicity, at each $\tau_l$. 
Hence (5.8) gives $\rho[u_{j(k)}, v_{j(k)}]\to \rho[u,v]$ at each $\tau_l$. But this implies convergence at each point of continuity 
of $\rho[u,v]$.
\end{proof}

Theorem 5.6 has an obvious consequence concerning geodesics.

\begin{thm}            %5.7
If $\varphi_j:[a,b]\to\cE$ are geodesics that decrease to a geodesic $\varphi:[a, b]\to\cE$, then $\rho_{\varphi_j}\to\rho_\varphi$ a.e.
\end{thm}

Indeed, $\rho_{\varphi_j}=\rho[\varphi_j(a), \varphi_j(b)]/(b-a)$ and $\rho_\varphi=\rho[\varphi(a),\varphi(b)]/(b-a)$.
This then proves Theorem 1.2.

The results of section 4 easily generalize from $\cH^{1\bar 1}$ to $\cE$. 

\begin{thm}        %5.8
Let $u,v,w\in\cE$. If $u\le v$ then $\rho[u,v]\ge \rho[u\wedge w, v\wedge w]$; if $u\ge v$ then $\rho[u,v]\le\rho[u\wedge w, v\wedge w]$.
\end{thm}
The statement follows from Lemma 4.2 and Theorem 5.6 upon representing $u,v,w$ as limits of decreasing sequences $u_j, v_j, w_j\in\cH^{1\bar 1}$.

\begin{thm}        %5.9
If $u,v,w\in\cE$ and $0\le\lambda<V$, then
\begin{equation}       %5.9
\int^\lambda_0\rho[u,v]+\int^\lambda_0\rho[v,w]\ge \int^\lambda_0\rho[u,w].
\end{equation}
\end{thm}

One should keep in mind that the integrals in (5.9) may be infinite. In this case a stronger version of the triangle inequality
$$
\int_0^{\lambda}\big(\rho(u,v)+\rho(v,w)-\rho(u,w)\big)\ge 0
$$
would be more meaningful; but we were not able to prove it. The difficulty is how to prove that the negative
part of the integrand is integrable near 0. 

In light of Theorem 5.6 the following is an immediate consequence of Lemma 3.2:

\begin{lem}     %5.10
Suppose $u,v,w\in\cE$.
\begin{itemize}
\item[(a)] If $v\le w$ then $\rho[u,v]\le \rho[u,w]$ and $\rho[v,u]\ge\rho[w,u]$.
\item[(b)] If $c\in\bR$, then $\rho[u, v+c]=\rho[u,v]+c=\rho[u-c, v]$.
\end{itemize}
\end{lem}

\begin{proof}[Proof of Theorem 5.9]
Observe that if $u_i, w_i\in\cE$ decrease to $u,w\in\cE$, then on $(0,\lambda)$ the functions $\rho[u_i, w_i]$ 
have a common lower bound. Indeed, if $\rho[u, w]$ is continuous at some $\tau\in(\lambda, V)$, then by 
Theorem 5.6 $\rho[u_i, w_i](\tau)$ is 
convergent, hence bounded below by some $c\in\bR$. Since $\rho[u_i, w_i]$ is a decreasing function, all $\rho[u_i, w_i]\ge c$ on $(0,\tau)\supset(0,\lambda)$. 
This has the consequence that Fatou's lemma applies and gives
\begin{equation}      % 5.10
\liminf_{i\to\infty}\int^\lambda_0\rho[u_i, w_i]\ge\int^\lambda_0\rho[u,w].
\end{equation}

%Observe further that (5.9) holds if $\int^\lambda_0\rho[u,v]=\infty$ or $\int^\lambda_0\rho[v,w]=\infty$. Therefore we will assume that both integrals on the left of (5.9) are finite.

Represent the given $u,v,w\in\cE$ as decreasing limits of $u_j,v_j,w_j\in\cE$ for which (5.9) holds, for example because they are in $\cH$, see Lemma 4.3. Using Lemma 5.10
\begin{equation}
\begin{aligned}   %5.11
\int_0^\lambda\rho[u_i,w] \le\int_0^\lambda\rho[u_i,w_k]&\le
\int_0^\lambda\rho[u_i,v_j]+\int_0^\lambda\rho[v_j,w_k]\\&\le
\int_0^\lambda\rho[u,v_j]+\int_0^\lambda\rho[v,w_k].
\end{aligned}
\end{equation}
The last two integrands decrease to
$\rho[u,v]$, $\rho[v,w]$, and are uniformly bounded above if $u,v$ happen to be bounded.
If we let $i,j,k\to\infty$, (5.9) then follows by (5.10) and by monotone convergence. This takes care of
the theorem if $u,v$ are bounded; $w$ can be arbitrary.

But this means that in our choice of $u_j,v_j,w_j$ we can take $w_j=w$. In the limit (5.11) then gives (5.9)
under the sole assumption that $u$ is bounded. Hence in (5.11) we can take $v_j=v$ as well; 
letting $i\to\infty$ we then obtain (5.9) in complete generality.  
\end{proof}

\section{Comparison with the flat geometry}    %section 6
The geometry of $\cE$ that we have studied so far, through its geodesics, originated from Mabuchi's connection on $\cH$. But 
$\cH$ has a simpler, flat geometry, too, inherited as an open subset of the Fr\'echet space $C^\infty(X)$. The geodesics in this geometry are straight line segments in $\cH\subset C^\infty(X)$, and just like Mabuchi's geodesics, the notion extends to 
$\psh(X,\omega)$ and $\cE$: for any $u,v$ in $\psh(X,\omega)$ or in $\cE$, the connecting segment is contained in the space, 
\cite[Proposition 1.6]{GZ1}. While the geodesics of Definition 1.1 are rather different from straight lines, it turns out that various quantities associated with the two are quite comparable. This is the theme we develop in this section. The reader again will notice connections with Darvas's work, in this case with his estimates of various distances in $\cE$ by simple integrals, 
\cite[Theorem 5.5]{Da1},  \cite[Theorem 6.1]{Da4}.

\begin{thm}      %6.1
If $u,v\in\cE$, then
\begin{equation}         %6.1
(v-u)^{\star v}\le\rho[u,v]\le(v-u)^{\star u}.
\end{equation}
If also $u\le v$, then
\begin{equation}         %6.2
\frac{(v-u)^{\star u}(s)}{n+1}\le\rho[u,v]\Big(\frac se\Big),\quad 0<s<V.
\end{equation}
\end{thm}

Here $v-u$ is defined and finite away from a pluripolar set, which is of $\mu_u$, $\mu_v$ measure 0. It is the velocity vector of 
the rectilinear path  $u+t(v-u)$, $0\le t\le 1$, joining $u$ and $v$. Thus the theorem compares the decreasing rearrangements of this velocity vector, when its footpoint is placed at one or the other endpoint, with the decreasing rearrangement of the velocity of the connecting geodesic.

\begin{lem} %6.?
(a) If $u_j,v_j\in\cE^\infty$ decrease to $u,v\in\cE^\infty$, then a.e.
\[
\lim_{j\to\infty}(v_j-u_j)^{\star u_j}=(v-u)^{\star u},\qquad \lim_{j\to\infty}(v_j-u_j)^{\star v_j}=(v-u)^{\star v}.
\]
(b) If $u,v\in\cE$ then a.e. (cf. (4.1))
\[
\limsup_{k\to-\infty}\limsup_{j\to-\infty}(v\vee j-u\vee k)^{\star u\vee k}\le(v-u)^{\star u},\quad
\liminf_{k\to-\infty}\liminf_{j\to-\infty}(v\vee k-u\vee j)^{\star v\vee k}\ge(v-u)^{\star v}.
\]
\end{lem}
\begin{proof}(a) \cite[Proposition 9.11]{GZ2} implies that $v_j-u_j\to v-u$ in capacity, whence  the claim follows by  
\cite[Lemma 3.4]{L2}.

(b) If $Y\subset X$ is Borel, by \cite[Proposition 10.5]{GZ2} $\mu_{u\vee k}(Y)\to\mu_u(Y)$ as $k\to -\infty$. 
This implies for every Borel function $\xi:X\to\bR$, defined a.e. with respect to $\mu_u,\mu_{u\vee k}$,
\begin{equation}       %6.3
\lim_{k\to -\infty} \xi^{\star u\vee k}=\xi^{\star u} \quad \text{a.e.}.
\end{equation}
Indeed, in $\lim_k\mu_{u\vee k}(\xi\ge \tau)=\mu_u(\xi\ge \tau)$ set 
$\tau=\xi^{\star u}(s-\varepsilon)+\varepsilon$ with  $\varepsilon\in(0,s)$. Since 
$\mu_u(\xi\ge\tau)\le\mu_u\big(\xi >\xi^{\star u}(s-\varepsilon)\big)\le s-\varepsilon < s$ by (2.3), for sufficiently negative $k$
\[
\mu_{u\vee k}(\xi\ge \tau)< s,\quad\text{and}\quad \xi^{\star u\vee k}(s)<\xi^{\star u}(s-\varepsilon)+\varepsilon
\]
by (2.4). Working with $\tau=\xi^{\star u} (s+\varepsilon)-\varepsilon$ and the set $(\xi>\tau)$, we obtain similarly that 
$\xi^{\star u\vee k}(s)>\xi^{\star u} (s+\varepsilon)-\varepsilon$ when $k$ is sufficiently negative. This proves 
the limit in (6.3) at points of continuity of $\xi^{\star u}$. Of course, (6.3) holds with $u$ replaced by $v$, too.

Now, if $j\le k$,
\[
(v\vee j-u\vee k)^{\star u\vee k}\le (v\vee j-u\vee j)^{\star u\vee k}, \quad 
(v\vee k-u\vee j)^{\star v\vee k}\ge(v\vee j-u\vee j)^{\star v\vee k}).
\]
Letting $j\to-\infty$, then $k\to-\infty$, the estimates follow by Lemma 2.1 and (6.3).
\end{proof}

\begin{proof}[Proof of Theorem 6.1]
Since all functions involved are left continuous, it suffices to prove (6.1), (6.2) on a dense subset of $(0,V)$.
(6.1) will be proved by reduction to the special cases $u,v\in\cH^{1\bar 1}$, respectively $u,v\in\cE^\infty$, and (6.2) will then be 
obtained, perhaps paradoxically, from (6.1).

We start by assuming $u,v\in\cH^{1\bar 1}$. Let $\varphi:[0,1]\to\cH^{1\bar 1}$ be the rectilinear path joining $u,v$, so 
$\varphi(t)=u+t(v-u)$, and $\psi:[0,1]\to \cH^{1\bar 1}$ the geodesic between $u$ and $v$. Since $\varphi(\cdot)(x)$ is linear,
$\psi(\cdot)(x)$ is convex for every $x\in X$, and the two agree at 0 and 1,
\[
\partial^+\psi(0)\le\partial^+\varphi(0)=v-u=\partial^-\varphi(1)\le\partial^-\psi(1).
\]
Passing to decreasing rearrangements we obtain (6.1).

Next assume $u,v$ are bounded. Choose $u_j, v_j\in\cH^{1\bar 1}$ that decrease to $u,v$; then
\[
(v_j-u_j)^{\star v_j}\le \rho[u_j, v_j]\le(v_j-u_j)^{\star u_j}.
\]
Hence Lemma 6.2a and Theorem 5.6 imply (6.1).

Finally, with general $u,v\in\cE$, by what has been proved already, if $j\le k$
\[
(v\vee k-u\vee j)^{\star v\vee k}\le \rho[u\vee j, v\vee k], \qquad
\rho[u\vee k, v\vee j]\le(v\vee j-u\vee k)^{\star u\vee k}.
\]
Letting $j\to -\infty$ then $k\to-\infty$, (6.1) follows from Lemmas 2.1, 6.2b and Theorem 5.6.

To prove (6.2), when $u\le v$ we introduce $w=(nu+v)/(n+1)$, the greedy version of a trick that in this context
goes back to proofs of 
\cite[Proposition 10.7]{GZ2} and \cite[Theorem 5.5]{Da1}. Thus $u\le w\le v$. We claim  (independently of the assumption 
$u\le v$)
\begin{equation}       %6.4
\mu_w\ge\Big(\frac{n}{n+1}\Big)^n\mu_u.
\end{equation}
If $u,v,w$ are smooth, this follows from
\[
\omega+i\partial\op w=\frac{n(\omega+i\partial\op u)}{n+1}+\frac{\omega+i\partial\op v}{n+1}\ge
\frac{n(\omega+i\partial\op u)}{n+1}.
\]
In general we choose $u_j, v_j\in\cH$ that decrease to $u,v$ as $j\to\infty$. Then $\mu_u, \mu_w$, as weak limits of the corresponding $\mu_{u_j}, \mu_{w_j}$, also satisfy (6.4).

(6.4) and (2.3)  imply for Borel functions $\xi:X\to\bR$, $0<s<V$, and $\varepsilon>0$
\[
\mu_w\big(\xi>\xi^{\star u}(s)-\varepsilon\big)\ge \mu_w\big(\xi\ge \xi^{\star u}(s)\big)\ge 
\Big(\frac{n}{n+1}\Big)^n \mu_u\big(\xi\ge\xi^{\star u}(s)\big)>\frac{s}{e}.
\]
Hence by (2.4)
$\xi^{\star w}(s/e)>\xi^{\star u}(s)-\varepsilon$, and so $\xi^{\star w}(s/e)\ge \xi^{\star u}(s)$.
Applying this with $\xi=w-u$, (6.1) and Lemma 5.10 yield 
\[
\frac{(v-u)^{\star u}(s)}{n+1}=(w-u)^{\star u}(s)\le (w-u)^{\star w}\Big(\frac{s}{e}\Big)\le
\rho[u,w]\Big(\frac{s}{e}\Big)\le\rho[u,v]\Big(\frac{s}{e}\Big),
\]
as claimed.
\end{proof}

\begin{cor}        %6.3
If $\varphi:[a,b]\to\cE(\omega)$ is a geodesic, then for $a\le t< b$, $a<\tau\le b$
\begin{equation}      %6.5
(\partial^-_\tau\varphi(\tau))^\star\le\rho_\varphi\le(\partial^+_t\varphi(t))^\star.
\end{equation}
If $\varphi$ increases $(\varphi(t)\le\varphi(\tau)$ when $t\le\tau$), then
\[(\partial^+_\tau\varphi(\tau))^\star(s)\le (n+1)\rho_\varphi(s/e),\quad 0<s<V.\]
\end{cor}

\begin{proof}
By Theorem 6.1, if $a\le\sigma<\tau$
\[
\Big(\frac{\varphi(\tau)-\varphi(\sigma)}{\tau-\sigma}\Big)^{\star\varphi(\tau)}=\frac{(\varphi(\tau)-
\varphi(\sigma))^{\star\varphi(\tau)}}{\tau-\sigma}\le\frac{\rho[\varphi(\sigma), \varphi(\tau)]}{\tau-\sigma}=\rho_\varphi.
\]
Letting $\sigma\to\tau$ we obtain the first inequality in (6.5) in light of Lemma 2.1. The other estimates are obtained similarly.
\end{proof}

\begin{cor}        %6.4
If $u,v\in\cE$ then $u\le v$ is equivalent to $\rho[u,v]\ge 0$.
\end{cor}

\begin{proof}
By Theorem 5.5 $\rho[u,v]\ge 0$ is equivalent to $\rho[u,u\wedge v]=0$. This latter certainly holds if $u\le v$. Conversely, suppose $\rho[u,u\wedge v]=0$. By Theorem 6.1
\[
0=(n+1)\rho[u\wedge v, u](s/e)\ge (u-u\wedge v)^{\star u\wedge v}(s).
\]
In other words, $u\le u\wedge v$ holds $\mu_{u\wedge v}$--a.e. According to S. Dinew, this implies 
$u\le u\wedge v$ everywhere, i.e. $u\le v$. Dinew does not seem to have published his proof, but communicated it to Levenberg, and the standard reference to the result is \cite[Proposition 5.9]{BL}.
\end{proof}

The following estimate we will need in section 10.
\begin{lem} %6.5
If $u,v,w\in\cE$ and $0<\sigma<s<V$, then
\begin{equation}  %6.6
\rho[u\wedge v, w](s)\le\max\big((w-u)^{\star u}(\sigma),\,(w-v)^{\star v}(s-\sigma)\big).
\end{equation}
\end{lem}
\begin{proof} As in Lemma 2.2, it suffices to prove for a dense set of $(\sigma,s)$. First assume $u,v\in\cH^{1\bar1}$ and $w\in\cE^\infty$. Since adding a constant to $w$ does not affect the validity of (6.6), we will also assume
$u,v\le w$. Define
\begin{equation*} 
\xi= (w-u\wedge v)1_{u\le v}, \qquad \eta=( w-u\wedge v)1_{u\ge v}. 
\end{equation*}
Thus $w-u\wedge v=\max(\xi,\eta)$. By the partition formula (4.6), if $\tau\ge 0$, 
\begin{align*}
\mu_{u\wedge v}(\xi>\tau) &=\mu_u(\xi>\tau, u\wedge v=u)+\mu_v(\xi>\tau, u> u\wedge v=v)\\
&=\mu_u(w-u\wedge v>\tau, u\wedge v=u)+0\le\mu_u(w-u>\tau).
\end{align*}
Let $s\in(0, V)$ and $\tau=(w-u)^{\star u}(s)$. With $\delta>0$ (2.3) implies
\[
\mu_{u\wedge v}(\xi\ge\tau+\delta)\le\mu_u\big(w-u>(w-u)^{\star u}(s)\big)\le s<s+\delta.
\]
Hence $\xi^{\star u\wedge v}(s+\delta)< (w-u)^{\star u}(s)+\delta$ by (2.4). Letting $\delta\to 0$, 
$\xi^{\star u\wedge v}(s)\le (w-u)^{\star u}(s)$ follows if $\xi^{\star u\wedge v}$ is continuous at $s$, 
so on a dense subset of  $(0,V)$. In fact
\[
\xi^{\star u\wedge v}\le(w-u)^{\star u}\qquad\text{and similarly}\qquad \eta^{\star u\wedge v}\le(w-v)^{\star v}
\]
everywhere, since the functions involved are left continuous. Lemma 2.2 implies
\begin{equation}
\begin{aligned}      %6.7
(w-&u\wedge v)^{\star u\wedge v}(s)= \max(\xi,\,\eta)^{\star u\wedge v}(s)\\
&\le \max\big(\xi^{\star u\wedge v}(\sigma),\,\eta^{\star u\wedge v}(s-\sigma)\big)
\le\max\big((w-u)^{\star u}(\sigma),\,(w-v)^{\star v}(s-\sigma)\big).
\end{aligned}
\end{equation}
In this form it is hard to carry over the estimate to $u,v\notin\cH^{1\bar1}$, but (6.7) implies (6.6)
in view of Theorem 6.1. Thus we proved (6.6) when $u,v\in\cH^{1\bar1}$, $w\in\cE^\infty$.

From this we obtain (6.6) for $u,v,w\in\cE^\infty$ by decreasing to $u,v$ with $u_j,v_j\in\cH^{1\bar1}$, and using Theorem 5.6 and Lemma 6.2a.
Thus for arbitrary $u,v,w\in \cE$ and $j,k\in\bZ$
\[
\rho[(u\vee k)\wedge (v\vee k), w\vee j](s)\le
\max\big((w\vee j-u\vee k)^{\star u\vee k}(\sigma),\,(w\vee j-v\vee k)^{\star v\vee k}(s-\sigma)\big).
\]
If we let $j\to -\infty$, then $k\to-\infty$, in the limit we obtain (6.6) by virtue of Theorem 5.6 and Lemma 6.2b.

\end{proof}

\section{The meaning of the rise}     %section 7

This section will be short. The meaning of the rise of a geodesic $\varphi:[a,b]\to\cH^{1\bar 1}$ is clear, it is 
$(\partial^\pm_t\varphi(t))^\star$ for any $t$, by definition. However, for a  geodesic in $\cE$ we do not understand the meaning 
of its rise directly. One could hope that the rise of a general geodesic can still be expressed in terms of velocities 
$\partial^\pm\varphi$, but we do not know how. Corollary 6.3 estimates $\rho_\varphi$ in terms of these velocities; this is the most we can say in general.

Still, if a geodesic $\varphi:[a,b]\to\cE$ passes through a single point in $\cH^{1\bar 1}$, its rise can be expressed by the velocities there.

\begin{thm}     % 7.1
Let $\varphi:[a,b]\to \cE$ be a geodesic. If $\varphi(\alpha)\in\cH^{1\bar 1}$ with some $\alpha\in[a,b]$, then
$\rho_\varphi=(\partial^\pm\varphi(\alpha))^\star$,
%\quad \text{(if $\alpha<b$)}\quad\text{and}\quad \rho_\varphi=(\partial^-\varphi(\alpha))^\star\quad \text{(if $\alpha>a$)}.\]
\end{thm}

\begin{proof}
We will prove when $\alpha=a$. The general result can be obtained from this special case by cutting the geodesic in two (and perhaps reversing it).

Construct geodesics $\varphi_j:[a,b]\to\cH^{1\bar 1}$ such that $\varphi_j(a)=\varphi(a)$ and $\varphi_j(b)$ decreases to 
$\varphi(b)$. By \cite[Proposition 3.15]{Da3} $\varphi_j(t)$ decreases to $\varphi(t)$ for all $t\in[a,b]$. If $a<s<t\le b$, 
by convexity
\begin{align*}
\frac{\varphi(s)-\varphi(a)}{s-a}&\le\frac{\varphi_j(s)-\varphi_j(a)}{s-a}\le \frac{\varphi_j(t)-\varphi_j(a)}{t-a}; 
\qquad \text{ letting } s\to a,\\
\partial^+\varphi(a)&\le \partial^+\varphi_j(a)\le \frac{\varphi_j(t)-\varphi_j(a)}{t-a}.
\end{align*}
Letting next $j\to\infty$, then $t\to a$ we obtain
\[\partial^+\varphi(a)\le \liminf_{j\to\infty} \partial^+\varphi_j(a)\le\limsup_{j\to\infty}\partial^+\varphi_j(a)\le\partial^+\varphi(a),\]
or $\partial^+\varphi(a)=\lim_j\partial^+\varphi_j(a)$. Thus, by Theorem 5.7 and Lemma 2.1
\[\rho_\varphi=\lim_j\rho_{\varphi_j}=\lim_j(\partial^+\varphi_j(a))^\star=(\partial^+\varphi(a))^\star,\]
first almost everywhere, but as a consequence, in fact everywhere on $(0,V)$.
\end{proof}

The following is immediate:
\begin{cor}        %7.2
If $\varphi:[a,b]\to\cE$ is a geodesic and $\varphi(\alpha)\in\cH^{1\bar 1}$ for some $\alpha\in(a,b)$, then 
$\partial^-\varphi(\alpha)=\partial^+\varphi(\alpha)$ holds $\mu_{\varphi(\alpha)}$--almost everywhere.
\end{cor}

A similar result follows from a recent paper by Di Nezza and Lu. A special case of \cite[Corollary 3.3]{DNL}
together with Corollary 6.3 above
imply that if $\varphi:[a,b]\to\cE^\infty$ is a geodesic and $\mu_{\varphi(\alpha)}$ has finite entropy for some
$\alpha\in[a,b]$, then $\rho_\varphi=\big(\partial^\pm\varphi(\alpha)\big)^\star$; if $\alpha\in(a,b)$, then
$\partial^-\varphi(\alpha)=\partial^+\varphi(\alpha)$ holds $\mu_{\varphi(\alpha)}$--a.e. 

\section{Lagrangians}			%section 8

The rest of the paper is concerned with Lagrangians on bundles larger than $T^\infty\cE$ and the relationship between rise and the induced action. This section will be about the structure and general properties of Lagrangians.

Our starting point is an invariant convex Lagrangian $L:T^\infty\cE\to\bR$ in the sense of Definition 2.4. \cite[Theorem 2.4]{L2} 
describes the structure of such Lagrangians:

\begin{thm}        %8.1
Given an invariant convex Lagrangian $L:T^\infty\cE\to\bR$, there is a family of $\cA_u\subset\bR\times B(X)$, $u\in\cH$, such that
\begin{equation}         %8.1
L(\xi)=\sup_{(a,f)\in\cA_u} a +\int_X\xi f \,d\mu_u, \quad \xi\in T_u\cH\approx C^\infty(X).
\end{equation}
The $\cA_u$ can be chosen invariant in the sense of the following definition.
\end{thm}
%We denote the space of Borel functions on $X$ by $\mathbb B(X)$ or by $\mathbb B(X,\mu_u)$, if the measure on $X$ is
%relevant. 

\begin{defn}      %8.2
Let $\cU\subset\cE$. A family of $\emptyset\ne\cA_u\subset\bR\times B(X)$, $u\in \cU$, is (strict rearrangement) 
invariant if whenever $u,v\in\cU$, $f\in B(X, \mu_u)$ and $g\in B(X, \mu_v)$ are equidistributed, and 
$(a,f)\in \cA_u$, then $(a,g)\in\cA_v$.  
\end{defn}

Since there are measure preserving bijections $(X, \mu_u)\to (X,\mu_v)$ for any $u,v\in\cE$ \cite[p.409, Theorem 16]{R}, an invariant family 
$\cA_u, u\in\cU$, is uniquely determined by any of its members $\cA_u$.

We will investigate functionals given by formulas like (8.1). Most of the material here will be needed in sections 9 and 10.

\begin{thm}        %8.3
Let $\emptyset\ne \cA_u\subset\bR\times B(X)$, $u\in\cE$, be an invariant family and 
\begin{equation}     %8.2
L(\xi)=\sup_{(a,f)\in\cA_u} a+\int_X\xi f \,d\mu_u\in(-\infty, \infty], \qquad \xi\in T^\infty_u\cE.
\end{equation}

(a)\ $L$ is finite on $T\cH$ if and only if
\begin{equation}     %8.3
\sup_{(a,f)\in\cA_u} a+\lambda \int_X |f| \,d\mu_u <\infty
\end{equation}
for all $\lambda\in(0,\infty)$ and for all (equivalently: for some) $u\in\cE$. In this case $L$ is finite, fiberwise convex and continuous in the sup norm topology on $T^\infty_u\cE\approx B(X)$.

(b)\ In addition to being finite on $T\cH$, $L$ is strongly continuous (Definition 2.4) on the fibers $T_u\cH$, $u\in\cH$, 
if and only if $\cA_u$ 
is uniformly integrable in the sense that for every $\varepsilon,\lambda>0$ there is  a $\delta>0$ such that whenever $u\in\cH$ and $E\subset X$ has $\mu_u$ measure $\le\delta$, then 
\begin{equation}     %8.4
\sup_{(a,f)\in\cA_u} a+\lambda \int_E |f| \,d\mu_u <L(0)+\varepsilon.
\end{equation}
In this case $L$ is strongly continuous on all fibers $T^\infty_u\cE$.
\end{thm}

Clearly, for any invariant family $\cA_u$, the Lagragian $L$ defined by (8.2) itself will be invariant under strict rearrangements.---We 
will need a result from \cite{L1}. Recall that given equidistributed  Borel functions $f$ on 
$(X,\mu_u)$ and $g$ on $(X,\mu_v)$, we write $(f,\mu_u)\sim (g,\mu_v)$, or just $f\sim g$ if the measures are understood. 
Another way to express equidistribution is $f^{\star u}=g^{\star v}$. If $\mu_u(Y)>0$ we write $\fint_Y\xi \,d\mu_u$ for the average  
$\int_Y\xi \,d\mu_u / \mu_u(Y)$. If $\mu_u(Y)=0$, we just set $\fint_Y\xi \,d\mu_u=0$.
Further, we put $\xi_+=\max(0,\xi)$, $\xi_-=\max(0,-\xi)$. The measure $\mu$ below is $\mu_0$; but it can be any $\mu_u$ 
because, as
said, all measures $\mu_u$ are isomorphic \cite[p.409, Theorem 16]{R}.

\begin{lem}[L1, Lemma 4.4, simplified]       %8.4
Let $f\in L^1(X,\mu)$, $\xi\in B(X)$, and $S, T\subset X$ of equal measure. If $\xi\ge 0$ on $T$ and $\xi\le 0$ on $X\setminus T$, then
\[
\sup_{g\sim f}\int_X\xi g\ge \fint_S f \,d\mu \int_X\xi_+ \,d\mu-\fint_{X\setminus S} f \,d\mu\int_X\xi_-\,d\mu.
\]
\end{lem}
A further result we will use concerns convergence properties of convex functions on general vector spaces.

\begin{lem}   %8.5
Let $\Xi$ be a real vector space and $l:\Xi\to(-\infty,\infty]$ convex. If $\xi,\xi_j\in \Xi$, $j\in \bN$, satisfy
\[
\limsup_{\alpha\to\pm\infty}\limsup_{j\to\infty} l\big((1-\alpha)\xi+\alpha\xi_j\big)<\infty,
\]
then $l(\xi)<\infty$ and $\lim_{j\to\infty}l(\xi_j)=l(\xi)$.
\end{lem}
\begin{proof} Write $\xi_j(\alpha)=(1-\alpha)\xi+\alpha\xi_j$, $\alpha\in\bR$, and note that $\xi_j(0)=\xi$, $\xi_j(1)=\xi_j$. 
Fix a positive number $C>\limsup_{\alpha\to\pm\infty}\limsup_{j\to\infty}l\big(\xi_j(\alpha)\big)$. Choose 
$\alpha_0>1$ and $j_\beta\in\bN$ for $\beta>\alpha_0$ such that $l\big(\xi_j(\alpha)\big)<C$ when $|\alpha|>\alpha_0$,
$j>j_{|\alpha|}$. Since  $l\big(\xi_j(\alpha)\big)$ is a convex function of $\alpha$, finite if $j>j_{|\alpha|}$, 
with such $j$ it
follows that $l(\xi)=l\big(\xi_j(0)\big)<\infty$. Convexity also implies with $j>j_{|\alpha|}$
\[
 l(\xi_j)-l(\xi)\,\begin{cases}\begin{aligned}
&\le\dfrac{l\big(\xi_j(\alpha)\big)-l(\xi)}\alpha<
 \dfrac{C-l(\xi)}\alpha\quad\text{if }\alpha>\alpha_0\\
&\ge \dfrac{l\big(\xi_j(\alpha)\big)-l(\xi)}\alpha>\dfrac{C-l(\xi)}\alpha\quad\text{if }\alpha<-\alpha_0.\end{aligned}\end{cases}
\]
Therefore $\limsup_{j\to\infty}|l(\xi_j)-l(\xi)|\le|C-l(\xi)|/|\alpha|$, which completes the proof.
\end{proof}

\begin{proof}[Proof of Theorem 8.3] It suffices to work with $u=0$, since all measure spaces $(X,\mu_u)$ are isomorphic. 
We continue writing $\mu$ for $\mu_0$. For simplicity we assume $\mu(X)=1$. We can also assume $L(0)=0$, since this can be 
arranged if we subtract $L(0)$ from all $a$ that occur in $\cA_u$. 

(a)\ If (8.3) holds, then with $\xi\in T^\infty_0\cH$ and $\lambda>\sup |\xi|$, 
\begin{equation}      %8.5
L(\xi)\le\sup_{(a,f)\in \cA_0} a+\lambda\int_X |f|\,d\mu_u <\infty
\end{equation}
shows $L$ is locally bounded above on $T^\infty_0\cH\approx B(X)$. As the supremum of affine functions, $L$ is convex; convexity and local upper boundedness imply continuity, see e.g. \cite[Lemma 4.2]{L1}.

Conversely, suppose $L$ is finite on $T\cH$. Given $\lambda\in(0,\infty)$, choose a nonnegative 
$\xi\in T_0\cH\approx C^\infty	(X)$ such that $\int_X\xi \,d\mu=\lambda$, but such that $T'=(\xi>0)$ has measure $\le 1/2$. Let $(a,f)\in \cA_0$, and assume first that $\cS=(f\ge 0)$ has measure $\ge 1/2$. Choose $T\supset T'$ with $\mu(S)=\mu(T)$. By (8.2) and by Lemma 8.4
\begin{equation}         %8.6
L(3\xi)\ge a+\sup_{g\sim f} 3\int_X \xi g\, d\mu\ge 
a+3\fint_S f \,d\mu\int_X\xi\, d\mu\ge a+3\lambda	\int_Xf_+\,d\mu.
\end{equation}

Viewing the constant functions $\pm 3\lambda$ as elements of $T^\infty_0\xi$, 
\begin{align*}
&L(3\lambda)\ge a+3\lambda\int_X f \,d\mu=a+3\lambda\int_X(f_+-f_-) \,d\mu,\\
&L(-3\lambda)\ge a-3\lambda\int_X f \,d\mu=a+3\lambda\int_X(f_--f_+) \,d\mu.
\end{align*}
Therefore
\begin{equation}        %8.7
a+\lambda\int_X|f|\,d\mu=a+\lambda\int_X(f_++f_-)\,d\mu \le\frac{2L(3\xi)+L(-3\lambda)}{3}.
\end{equation}
If, instead of $(f\ge 0)$, the set $(f\le 0)$ has measure $\ge 1/2$, we apply (8.6) with $\xi$ replaced by $-\xi$, to obtain 
\begin{gather}
L(-3\xi)\ge a+\sup_{g\sim -f} 3\int_X \xi g \,d\mu\ge a +3\lambda\int_X(-f)_+ \,d\mu=
a+3\lambda\int_X f_-\,d\mu,
\quad\text{and}\nonumber \\
 %8.8
a+\lambda\int_X|f|\,d\mu=a+\lambda\int_X(f_++f_-)\,d\mu \le\frac{2L(-3\xi)+L(3\lambda)}{3}.
\end{gather}
To sum up, for every $(a,f)\in \cA_0$ either (8.7) or (8.8) holds, and so
\begin{equation}        %8.9
\sup_{(a,f)\in\cA_0} a+\lambda\int_X|f|\,d\mu\le
\max\Big(\frac{2L(3\xi)+L(-3\lambda)}{3}, \frac{2L(-3\xi)+L(3\lambda)}{3}\Big) <\infty.
\end{equation}

(b)\ Suppose $\cA_0$ is uniformly integrable in the sense of (8.4). Strong continuity of $L|T^\infty_0\cE$ will follow from 
Lemma 8.5.  We start by deriving a better bound on $L$ than (8.5).

If $\lambda,\delta\in (0, \infty)$, let
\begin{equation}        %8.10
C(\lambda,\delta)=\sup\Big\{a+\lambda\int_E|f| \,d\mu : (a,f)\in \cA_0, \mu(E)\le\delta\Big\}.
\end{equation}
Our assumption is $\lim_{\delta\to 0} C(\lambda,\delta)=0$. Let $\eta\in T^\infty_0\cE$,  $\lambda>0$. 
If $(a,f)\in \cA_0$
\begin{align}
2\Big(a+\int_X \eta f \,d\mu\Big)&
\le a+2\sup |\eta|\int_{|\eta|>\lambda} |f| \,d\mu +a+2\lambda\int_{|\eta\le\lambda} |f|\,d\mu\nonumber \\
&\le C\big(2\sup |\eta|, \mu({|\eta|>\lambda})\big)+ C(2\lambda, 1), \qquad\text{ whence}\nonumber\\
        %8.11
2L(\eta)&\le C\big(2\sup |\eta|, \mu(|\eta|>\lambda)\big)+C(2\lambda, 1).
\end{align}
This can be an improvement on (8.5), which says $L(\eta)\le C(\sup |\eta|, 1)$.

Now consider uniformly bounded $\xi_j\in T^\infty_0\cE$,  $j\in \bN$, that $\mu$--a.e. tend to $\xi\in T^\infty_0\cE$. 
To show $L(\xi_j)\to L(\xi)$, with $\alpha\in\bR$ let
\[
\xi_j(\alpha)=(1-\alpha)\xi+\alpha\xi_j; \qquad \lim_{j\to\infty}\xi_j(\alpha)=\xi\quad \mu\text{--a.e.}
\]
Let $\lambda= 1+\sup_j\sup_X |\xi_j|$ and note that $\sup_X |\xi_j(\alpha)|< (2|\alpha|+1)\lambda$. 
Given $\alpha$, choose $\delta>0$ so that $C\big((4|\alpha|+2)\lambda, \delta\big)<1$.
By Egorov's theorem $\lim_{j\to\infty}\xi_j(\alpha)=\xi$ uniformly outside a set of measure $\le\delta$. Choose $j_0$ such that
\[
\mu(|\xi_j(\alpha)|>\lambda)\le\delta \quad\text{for}\quad j>j_0.
\]
For such $j$, by (8.11)
\[
2L(\xi_j(\alpha))\le C\big((4|\alpha|+2)\lambda,\delta\big)+C(2\lambda, 1)< 1+C(2\lambda, 1).
\]
Lemma 8.5 applies, and gives $L(\xi_j)\to L(\xi)$. We conclude $L$ is strongly continuous on $T_0^\infty\cE$.

To prove the converse implication we assume $L$ is strongly continuous on $T_0\cH$. It will be convenient to pass to decreasing rearrangements of $f,\xi\in B(X,\mu)$; then (8.2) can be rewritten
\[
L(\xi)=\sup_{(a,f)\in \cA_0} a+\int^1_0\xi^\star f^\star,\quad \xi\in T_0^\infty\cE,
\]
the integral with respect to Lebesgue measure. Indeed, if $\theta:(X,\mu)\to \big((0,1)$,Lebesgue\big) preserves measure, then e.g. 
\cite[Lemma 7.2]{L1} implies
\[
\sup_{g\sim f}\int_X \xi f\,d\mu=\int_X (\xi^\star\circ\theta)(f^\star\circ\theta)\,d\mu=\int^1_0\xi^\star f^\star.
\]
Accordingly consider $M: B(0,1)\to\bR$ given by
\[
M(\zeta)=\sup_{(a,f)\in\cA_0} a+\int^1_0\zeta f^\star.
\]
Thus $M(\zeta)=L(\zeta\circ\theta)$, and part (a) implies that $M$ is continuous if $B(0,1)$ is endowed with the supremum norm. We show that it is even strongly continuous on $C[0,1]$: if uniformly bounded $\zeta_j\in C[0,1]$ converge almost everywhere, then $M(\zeta_j)$ also converges.

For this purpose we need a continuous measure preserving $\theta:(X,\mu)\to [0,1]$. E.g. by \cite[Lemma 5.5]{L1} such 
$\theta$ exists. Thus $\zeta_j\circ\theta\in C(X)$. Choose $\xi_j\in C^\infty(X)\approx T_0\cH$ such that 
\[
\sup_X|\xi_j-\zeta_j\circ\theta|<1/j, \qquad |L(\xi_j)-L(\zeta_j\circ\theta)|<1/j.
\]
Then $\zeta_j\circ\theta$ and $\xi_j$ converge $\mu$--a.e. Hence $L(\xi_j)$ converge and so do $L(\zeta_j\circ \theta)=M(\zeta_j)$. 

We can also rewrite $C(\lambda,\delta)$ of (8.10) as
\begin{equation} %8.12
C(\lambda,\delta)=\sup_{(a,f)\in \cA_0} a+\lambda\int^\delta_0 |f^\star|,
\end{equation}
and we need to show $\lim_{\delta\to 0} C(\lambda,\delta)=0$ for all $\lambda >0$. Part(a) at least implies
\[
0=L(0)=\sup_{(a,f)\in\cA_0} a\le C(\lambda,\delta)\le C(\lambda, 1)<\infty.
\]
For any $(a,f)\in \cA_0$, the negative part $f^\star_-$ of $f^\star$ is increasing. Therefore by (8.12)
\begin{gather}       %8.13
C(\lambda, 1)\ge  a+\lambda\int^1_{1/2} f^\star_-\ge
a+\frac{\lambda}{2} f^\star_-\Big(\frac {1}{2}\Big),\qquad \text{whence}\nonumber \\  
\lambda f^\star_-(\tau)\le2\big(C(\lambda,1)-a\big)\quad \text{if}\quad 0<\tau<1/2.  
\end{gather}
For $j=2,3,\dots$ consider the functions $\zeta_j\in C[0,1]$,
\[
\zeta_j=\begin{cases}
\lambda \quad\text{on}\quad [0,1/j]\\
0 \quad \text{on}\quad [2/j, 1] \\
\text{linear in between,}
\end{cases}
\]
uniformly bounded and tending to $0$ a.e. as $j\to\infty$; whence $M(\zeta_j)\to M(0)=0$. If $(a,f)\in\cA_0$ and $j\ge 4$, using 
(8.13) and that $L(0)=0$ implies $ a\le 0$,
\begin{align*}
M(\zeta_j)&\ge a+\int^{1}_0 \zeta_j f^{\star}=a+\int^{2/j}_0 \zeta_j\big( |f^{\star}|-2f^\star_-\big)
\ge a+\int_0^{2/j}\zeta_j|f^\star|-2\lambda\int_0^{2/j}f^\star_-\\
&\ge a+\lambda\int^{1/j}_0  |f^{\star}|-\frac{8}j\big(C(\lambda, 1)-a\big)
\ge 3\Big(a+\frac{\lambda}{3}\int^{1/j}_0|f^\star|\Big)-\frac{8C(\lambda, 1)}{j}.
\end{align*}
In view of (8.12), therefore
\[
0\le 3C(\lambda/3, 1/j)\le M(\zeta_j)+8C(\lambda, 1)/j\to 0, \quad j\to\infty,
\]
which completes the proof of Theorem 8.3.
\end{proof}

Now consider an invariant family of nonempty $\cA_u\subset\bR\times B(X)$, $u\in\cE$. Such a family defines a function 
$L: T^\infty\cE\to (-\infty,\infty]$ by formula (8.2), and in fact on a larger (set theoretical) Banach bundle $T^1\cE\to\cE$. 
The fibers of this latter are
\[
T^1_u\cE=L^1(X,\mu_u), \quad\text{and} \quad T^1\cE=\coprod_{u\in\cE} T^1_u\cE.
\]
Occasionally it will be convenient to view $T^\infty\cE$ as a subbundle of $T^1\cE$. Although in reality it is only a quotient of 
$T^\infty_u\cE$, modulo equality $\mu_u$--a.e., that is a subspace of $T^1_u\cE$, this should not cause confusion, since all the 
Lagrangians we work with here take the same value on functions that agree a.e.---The invariant family $\cA_u$ thus determines $L:T^1\cE\to (-\infty,\infty]$, 
\begin{equation}           %8.14
L(\xi)=\sup_{(a,f)\in\cA_u} a+\int_X \xi f \,d\mu_u, \qquad \xi\in T^1_u\cE,
\end{equation}
a fiberwise lower semicontinuous convex function that is strict rearrangement invariant. (The topology on $T_u^1\cE$ is the
$L^1$ topology.) With $\cA_u$ we can associate three more invariant families
\begin{equation}\begin{gathered}          %8.15
\cA^+_u=\{(a,f_+): (a,f)\in \cA_u\}, \quad \cA^-_u=\{(a,f_-): (a,f)\in \cA_u\},\\  
\cA^{|\, |}_u=\{(a,g)\in\bR\times B(X): \text{ there is }(a,f)\in\cA_u \text{ such that } |g|=|f|\},
\end{gathered}\end{equation}
and Lagrangians $L^+, L^-, L^{|\, |}: T^1\cE\to (-\infty,\infty]$,
\begin{equation}\begin{gathered}          %8.16
L^\pm(\xi)=\sup_{(a,g)\in \cA^\pm_u} a+\int_X\xi g \,d\mu_u=\sup_{(a,f)\in\cA_u} a+\int_X\xi f_\pm \,d\mu_u,\\  
L^{|\, |}(\xi)=\sup_{(a,g)\in \cA^{|\,  |}_u} a+\int_X\xi g \,d\mu_u=\sup_{(a,f)\in\cA_u} a+\int_X|\xi f| \,d\mu_u,
\end{gathered}\end{equation}
$\xi\in T^1_u\cE$. These Lagrangians are comparable:

\begin{thm}          %8.6
If $\xi\in T^1_u\cE$ then
\begin{equation}\begin{gathered}          %8.17
2L^+(\xi)\le L(2\xi)+L^-\Big(2\fint_X\xi \,d\mu_u\Big), \quad 2L^-(\xi)\le L(-2\xi)+L^+\Big(-2\fint_X\xi \,d\mu_u\Big),\\ 
2L^{|\, |}(\xi)\le\max\big(L(8\xi), L(-8\xi)\big)+L^{|\, |}\Big(6\fint_X |\xi|\,d\mu_u\Big),\\
|L(\xi)-L(0)|\le L^{|\, |}(\xi)-L^{|\, |}(0). 
\end{gathered}\end{equation}
Furthermore, if $L | T^\infty\cE$ is finite and strongly continuous on the fibers $T_u^\infty\cE$, then the same holds for $L^\pm, L^{|\, |}$.
\end{thm}

In (8.17) averages such as $\fint_X\xi \,d\mu_u=\int_X\xi \,d\mu_u / V$  are  constant functions on $X$, and are viewed as $\in T_u^1\cE$. 
The interest of the theorem is that in various situations it allows one to replace $L$ by $\Lambda=L^{|\, |}$, which is absolutely monotone in the following sense:
\begin{defn}       %8.7
A function $\Lambda: T^1\cE\to (-\infty, \infty]$ is absolutely monotone if $\xi,\eta\in T^1_u\cE$, $|\xi|\le|\eta|$ imply 
$\Lambda(\xi)\le \Lambda(\eta)$.
\end{defn}
\begin{proof}
The proof depends on certain estimates proved in \cite{L1}.\footnote{\cite{L1} typically works with bounded
functions, and for this reason the results we cite in this proof and later directly apply only to bounded 
$\xi$. Nonetheless,
the estimates follow for general $\xi\in L^1(X,\mu_u)$ if we replace $\xi$ by $\max\big(\min(\xi, c),-c\big)$,
then let the constant $c\to \infty$.} Again, it suffices to work with $u=0$; we write $\mu_0=\mu$. Let 
$(a,f_0)\in\cA_0$. By \cite[(7.4)]{L1}
\begin{align*}          
2a+2\sup_{f\sim f_0} \int_X\xi f_+&\le 
a+2\sup_{f\sim f_0} \int_X \xi f \,d\mu+a+2\int_X f_{0-} \,d\mu\fint_X \xi \,d\mu\\
&\le L(2\xi)+L^-\Big(2\fint_X\xi \,d\mu\Big).
\end{align*}
Passing to the supremum over all $(a,f_0)\in\cA$ gives the first estimate in (8.17). The second estimate follows by the same computation, with $f$ replaced by $g=-f$.

To prove the third estimate we use \cite[Lemma 7.5]{L1}: 
\[
\sup_{f\sim f_0} \int_X |\xi f| \,d\mu \le 4\sup_{f\sim f_0}\Big |\int_X\xi f \,d\mu\Big|+3\fint_X|\xi|\,d\mu\int_X |f_0| \,d\mu.
\]
Therefore
\begin{equation}         %8.18
2\sup_{f\sim f_0}a+ \int_X |\xi f| \,d\mu \le 
\sup_{f\sim f_0}\Big(a+ \Big|\int_X 8\xi f \,d\mu\Big|\Big)+a+6 \fint_X|\xi|\,d\mu \int |f_0| \,d\mu. 
\end{equation}
Now
\[
a+\Big|\int_X8\xi f \,d\mu\Big|=\max \Big(a+\int_X 8\xi f \,d\mu, a-\int_X 8\xi f \,d\mu\Big)\le \max \big(L(8\xi), L(-8\xi)\big);
\]
taking the supremum over all $(a,f_0)\in \cA_0$ in (8.18) thus gives the third estimate in (8.17). As to the last estimate, from the definition $L(\xi)\le L^{|\, |}(\xi)$ and $L(\xi)-L(0)\le L^{|\,|}(\xi)-L^{|\,|}(0)$. By convexity $2L(0)\le L(\xi)+L(-\xi)$, whence
\[
L(0)-L(\xi)\le L(-\xi)-L(0)\le L^{|\, |}(-\xi)-L(0)=L^{|\, |}(\xi)-L^{|\, |}(0),
\]
and so indeed $|L(\xi)-L(0)|\le L^{|\, |}(\xi)-L^{|\, |}(0)$.

The last claim of Theorem 8.6 follows from Theorem 8.3b, as $L^\pm(0)=L^{|\, |}(0)=L(0)$.
\end{proof}

There is another way to introduce the class of Lagrangians of (8.14).

\begin{lem}       %8.8
A function $L:T^1\cE\to (-\infty,\infty]$ can be represented in the form (8.14) with an invariant family 
$\cA_u\subset \bR\times B(X)$, $u\in\cE$, if and only if it is strict rearrangement invariant, and on the fibers $T^1_u\cE$ it 
is convex and lower semicontinuous.
\end{lem}

\begin{proof}
If $L$ is representable in the form (8.14), then as the supremum of continuous affine functions on $T^1_u\cE$, it is convex and lsc. If the family $\cA_u$, $u\in\cE$, is invariant and $\xi\in T^1_u\cE(\omega)$, then for every $(a,f_0)\in\cA_u$
\[
a+\sup_{(f,\mu_u)\sim(f_0,\mu_u)} \int_X\xi f \,d\mu_u=a+\int^V_0\xi^{\star u} f^{\star u},
\]
since the sup above occurs when $f$ and $\xi$ are `similarly ordered', see e.g. 
\cite[Lemma 7.2]{L1}, whence $\xi f$ and $\xi^{\star u} f^{\star u}$ are equidistributed. Hence
\[
L(\xi)=\sup_{(a,f)\in\cA_u} a +\int^V_0\xi^{\star u} f^{\star u},
\]
obviously invariant under strict rearrangements.

Conversely, suppose $L$ is strict rearrangement invariant, and convex and lsc on the fibers $T^1_u\cE$. Let $L^*$ denote its fiberwise conjugate convex function (Fenchel--Legendre transform)
\[
L^*(f)=\sup\Big\{\int_X\xi f \,d \mu_u-L(\xi): \xi\in T^1_u\cE\Big\},\qquad f\in B(X, \mu_u).
\]
By Fenchel's theorem \cite[p.175]{IT}
\[
L(\xi)=\sup\Big\{\int_X\xi f \,d\mu_u -L^*(f): f\in B(X,\mu_u)\Big\},\qquad\xi\in T^1_u\cE,
\]
and so
\begin{equation}           %8.19
\cA_u=\big\{\big(-L^*(f),f\big): f\in B(X,\mu_u)\big\}
\end{equation}
gives the required representation (8.14).
\end{proof}

\section{The bundle $T^L\cE\to\cE$}      %Section 9

With the Lagrangians $L:T^1\cE\to(-\infty, \infty]$ studied in the previous section here we will associate a subbundle 
$T^L\cE\subset T^1\cE$, and through it, in the next section, an energy space $\cE^L\subset\cE$. In the simplest cases, 
such as $L(\xi)=\int_X|\xi|^p\,d\mu_u$, the subbundle consists of $\xi$ for which $L(\xi)$ is finite. In general the 
definition of $T^L\cE$ has to be more involved for two reasons. First, $L(\xi)<\infty$ in general does not imply $L(\alpha\xi)<\infty$ 
for values of $\alpha\in\bR$ other than $\alpha\in [0,1]$. Second, to work efficiently with $L$ we need to restrict it to a subbundle 
on which it has a certain fine continuity property.

\begin{defn}      %9.1
Consider a function $L:T^1\cE\to (-\infty,\infty]$ with $L(0)$ finite. It is absolutely continuous on $\xi\in T^1_u\cE$ if for every 
$\varepsilon>0$ there is a $\delta>0$ such that
\begin{equation}         %9.1
|L(\xi 1_E)-L(0)|<\varepsilon\quad\text{whenever}\quad \mu_u(E)\le\delta.
\end{equation}
If $\Xi\subset T^1\cE$, we say $L$ is absolutely equicontinuous on $\Xi$ if for every $\varepsilon>0$ there is a $\delta>0$ such that (9.1) holds whenever $u\in\cE$, $\xi\in\Xi\cap T^1_u\cE$, and $\mu_u(E)\le\delta$.
\end{defn}

The term `absolutely continuous' is borrowed from the theory of rearrangement invariant Banach spaces, see e.g. 
\cite[p.13]{BS}. As an example, $L(\xi)=\int_X|\xi|^p\,d\mu_u$, $1\le p<\infty$, is absolutely continuous on $\xi$ if and only if $L(\xi)<\infty$; but $L(\xi)=\text{ess\,sup}_X|\xi|$ is absolutely continuous only on $\xi=0$. In general, if $\xi,\eta\in T^1\cE$ are equidistributed,  and an invariant $L$ is absolutely continuous on $\xi$, it is absolutely continuous on $\eta$, too.

Henceforward we will work with strict rearrangement invariant Lagrangians $L:T^1\cE\to(-\infty,\infty]$ that are convex and lsc on the fibers, and we assume $L$ is finite on $T^\infty\cE$. Starting with such $L$, we can represent it through the family $\cA_u$ constructed in (8.19); then we can form $\cA_u^{|\, |}$ and 
$L^{|\, |}$ of (8.15), (8.16). Because of Theorem 8.3a, $L^{|\, |}$ is also finite on $T^\infty\cE$. It is of course strict rearrangement invariant, and fiberwise convex and lsc. 

Now to the definition of $T^L\cE$. Whether to include $\xi\in T^1\cE$ into $T^L\cE$ will depend not only on how $L$ interacts with $\xi$, but also on how it interacts with $\eta$ that are allied to $\xi$, in the following sense.

\begin{defn}    %9.2
We say that $\eta\in T^1_v\cE$ is allied to $\xi\in T^1_u\cE$ if there are $\alpha\in\bR$ and $\beta\in(0, 1]$ such that $\eta^{\star v}(s)=(\alpha\xi)^{\star u}(\beta s)$ for $0<s<V$.
\end{defn}

If this relation holds, we will say $\eta^{\star v}$ is the $(\alpha,\beta)$-rescaling of $\xi^{\star u}$.---In particular, $\alpha=\beta=1$ corresponds to equidistribution.

\begin{defn}       %9.3
We let $T^L_u\cE$ consist of $\xi\in T^1_u\cE$ such that $L$ is finite and absolutely continuous on any $\eta\in T^1\cE$ (or only on any $\eta\in T^1_u\cE$) allied to it. Further, we let $T^L\cE=\coprod_{u\in\cE} T^L_u\cE$, a subbundle of $T^1\cE\to\cE$.
\end{defn}

If $L(\xi)=\int_X\chi(\xi) \,d\mu_u$ with some convex function $\chi:\bR\to\bR$, then the definition of $T^L\cE$ simplifies to requiring $L(\alpha\xi)<\infty$ for all real $\alpha$; and if $\chi$ is quasihomogeneous, $\chi(t)/c \le\chi(\pm 2t)\le c\chi(t)$ with some $c>0$, then simply $T^L\cE=\{\xi\in T^1\cE: L(\xi)<\infty\}$. 

One could also associate with $L$ another, in general larger subbundle $T^{(L)}\cE\to\cE$, consisting of $\xi\in T^1\cE$ for which there is an $\alpha_0>0$ such that $L$ is finite and absolutely continuous on $\eta$ whenever 
$\eta^\star(s)=(\alpha\xi)^{\star}(\beta s)$ with $\alpha\in(-\alpha_0,\alpha_0)$, $\beta\in(0, 1]$. The theory would be similar, but we will not pursue it here. 

\begin{lem}    %9.4
(a)\  $T^L_u\cE\subset T^1_u\cE$ is a vector subspace.\newline
(b)\ If $\xi\in T^1_u\cE$, $\xi', \xi''\in T^L_u\cE$, and $\xi'\le\xi\le\xi''$ $\mu_u$--a.e., then $\xi\in T^L_u\cE$ and $|\xi|\in T^L_u\cE$. \newline
(c)\ If $\Lambda=L^{|\, |}$, then $T^{\Lambda}\cE=T^L\cE$.
\end{lem}
\begin{proof}
If we add a constant to $L$, the bundles $T^L\cE, T^\Lambda\cE$ are not going to change. For this reason we can assume $L(0)=0=\Lambda(0)$. (c) now follows from the estimates in Theorem 8.6. Therefore in the rest of the proof, at the price of replacing $L$ by $L^{|\, |}$, we can assume that $L$ is absolutely monotone (Definition 8.7). In particular, $L(\xi)=L(|\xi|)\ge 0$. 

To prove (b) it now suffices to show that if $\xi,\xi', \xi''\in T^1_u\cE$, $L$ is finite and absolutely continuous on 
$\alpha\xi', \alpha\xi''$ for all $\alpha\in\bR$, and $\xi'\le\xi\le\xi''$, then $L$ is finite and absolutely continuous on $\xi$. But this follows since $|\xi|\le|\xi'| +|\xi''|$, and so with any Borel set $E\subset X$
\[
L(\xi 1_E)\le L(|\xi'| 1_E+|\xi''| 1_E)\le \big(L(2\xi' 1_E)+L(2\xi'' 1_E)\big)/2.
\]

Finally, to prove (a) we need to show that if $\xi, \eta\in T^L_u\cE$, and $\zeta\in T^1_u\cE$ is allied to 
$\xi+\eta$, then $L$ is 
finite and absolutely continuous on $\zeta$ or, equivalently, on $|\zeta|$. Suppose 
$\zeta^\star(s)=\big(\alpha(\xi+\eta)\big)^\star (\beta s)$. Upon multiplying $\xi, \eta$ with a suitable unimodular function,  we can arrange that $\zeta^\star\ge 0$. Then
\[
\zeta^\star(s)\le|\alpha|(|\xi|+|\eta|)^\star (\beta s)\le |\alpha\xi|^\star (\beta s/2)+|\alpha\eta|^\star (\beta s/2),
\]
the last inequality by Lemma 2.2, say. 
Choose $\xi_1,\eta_1\in T^1_u\cE$ such that 
$\xi_1^\star (s)=|\alpha\xi|^\star (\beta s/2)$, $\eta_1^\star(s)=|\alpha \eta|^\star (\beta s/2)$. These $\xi_1, \eta_1$ can be 
obtained by pulling back e.g. $\alpha|\xi|^\star(\beta s/2)$ by a measure preserving $\theta:(X,\mu_u)\to (0, V)$. Let $\zeta_1$ 
be the pull back of $\zeta^\star$ by the same $\theta$. Then $0\le\zeta_1\le\xi_1+\eta_1$. Since  with any Borel set 
$E\subset X$
\[
L(\zeta_1 1_E)\le L(\xi_1 1_E+\eta_1 1_E)\le \big(L(2\xi_1 1_E)+L(2\eta_1 1_E)\big)/2,
\]
$L$ is finite and absolutely continuous on $\zeta_1$, hence on the equidistributed $\zeta$ as well.
\end{proof}

In $T^L\cE$ a version of the Dominated Convergence Theorem holds:

\begin{lem}         %9.5
If $\xi_j, \eta_1, \eta_2\in T^L_u\cE$ for $j=1, 2,\dots$, $\eta_1\le\xi_j\le\eta_2$, and $\xi_j\to\xi$
holds $\mu_u$--a.e., then $L(\xi_j)\to L(\xi)$.
\end{lem}
\begin{proof}
Let $\xi_j(\alpha)=(1-\alpha)\xi+\alpha\xi_j$, $\alpha\in\bR$. 
The claim will follow from Lemma 8.5 once we show that
\begin{equation}             %9.2
\sup_{\alpha\in\bR}\limsup_{j\to\infty} |L(\xi_j(\alpha))|<\infty.
\end{equation}

At the price of replacing $L$ by $L^{| \, |}-L(0)$, we can assume $L$ is absolutely monotone and $L(0)=0$, cf. Theorem 8.6,
Lemma 9.4. By this lemma $\eta=|\eta_1|+|\eta_2|\in T^L_u\cE$. Given $\alpha$, choose $\delta>0$ so that 
$L((4\alpha+2)\eta 1_E)<1$ 
when $\mu_u(E)\le\delta$, and choose such an $E$ outside which $\xi_j$ converges uniformly. By convexity
\begin{equation}             %9.3
0\le 4L\big(\xi_j(\alpha)\big)\le 
2L\big(2\xi_j(\alpha) 1_E\big)+L\big(4(\xi_j(\alpha)-\xi) 1_{X\setminus E}\big)+L(4\xi 1_{X\setminus E}).
\end{equation}
Since $|\xi_j(\alpha)|\le (2|\alpha|+1)\eta$, the first term on the right is $\le 2L\big((4|\alpha|+2)\eta 1_E\big)<2$. 
The second term tends to 0 as $j\to\infty$ for the following reason. The function $\bR\ni\lambda\mapsto L(\lambda)\in\bR$ 
is convex, hence continuous. Therefore letting $\lambda_j=\sup_{X\setminus E} 4|\xi_j(\alpha)-\xi|$
 \[
 L\big(4(\xi_j(\alpha)-\xi) 1_{X\setminus E}\big)\le L(\lambda_j)\to 0\quad\text{as } j\to\infty,
 \]
since $\xi_j(\alpha)\to\xi$ uniformly on $X\setminus E$. Finally, the last term in (9.3) is bounded by $L(4\xi)$. Hence $\limsup_j 4L(\xi_j(\alpha))\le 2+L(4\xi)$, and (9.2) indeed holds.
\end{proof}

Absolute continuity of Definition 9.1 and strong continuity of Definition 2.4 are related.

\begin{lem} The following are equivalent:      %9.6
\item{(i)}\ $L|T^\infty\cE$ is strongly continuous;
\item{(ii)}\ $T^\infty\cE\subset T^L \cE$;
\item{(iii)}\  $1\in T^L\cE$.
\end{lem}

\begin{proof}

\ (i)$\Rightarrow$ (ii). We need to show that if $\xi\in T^\infty_u\cE$ and $E_j\subset X$ satisfy 
$\mu_u(E_j)\to 0$, then $L(\xi 1_{E_j})\to L(0)$. If this failed, we could find $\varepsilon>0$ and $E_j$ such that $\sum^\infty_{j=1}\mu_u (E_j)<\infty$ and $|L(\xi 1_{E_j})-L(0)|>\varepsilon$. But $\xi_j=\xi 1_{E_j}$ are uniformly bounded and tend $\mu_u$--a.e. to 0, hence $L(\xi_j)\to L(0)$, a contradiction.

 (ii)$\Rightarrow$ (i) is a special case of Lemma 9.5, and
 (ii) $\Leftrightarrow$ (iii) is partly obvious, partly follows from Lemma 9.4.
\end{proof}
{\sl Henceforward we assume $L$ satisfies the equivalent conditions of Lemma 9.6.}

Our Lagrangian $L$ induces a functional $L_\star$ on $L^1(0,V)$. If $u\in\cE$ and 
$\theta:(X,\mu_u)\to \big((0, V)$, Lebesgue\big) preserves measure, then with $\zeta\in L^1(0,V)$ (cf. (2.12))
\begin{equation}             %9.4
L_\star(\zeta)=L(\zeta\circ\theta)\in(-\infty,\infty],\quad \zeta\circ\theta \text{ viewed as } \in T^1_u\cE.
\end{equation}

We denote by $\cE^L(0,V)\subset L^1(0,V)$ the subspace of those $\zeta$ for which $\zeta\circ\theta\in T^L\cE$. The invariance 
of $L$ implies that $L_\star$ and $\cE^L(0, V)$ are independent of the choice of $u$ and $\theta$. We have for $\xi\in T^1\cE$
\[
L(\xi)=L_\star(\xi^\star),\quad\text{and}\quad \xi\in T^L\cE\quad\text{if and only if}\quad \xi^\star\in\cE^L(0,V).
\]

Part (b) of the next result suggests that $T^L\cE$ is complete in some sense:

\begin{lem}        %9.7
(a)\ Let $\xi_j\in T^1_u\cE$, $j\in\bN$. If $L$ is absolutely continuous on $\alpha\xi_j$ and $\lim_{i,j\to\infty} L(\alpha(\xi_i-\xi_j))=L(0)$ for all $\alpha\in\bR$, then $L$ is absolutely equicontinuous on the family $\{\xi_j: j\in\bN\}$.

(b)\ Let $\xi_j\in T^L_u\cE$, $j\in\bN$. Suppose for every $\alpha\in\bR$, $\beta\in(0,1]$ and 
$\zeta_j(s)=\alpha\xi^\star_j(\beta s)$, $s\in (0, V)$, we have $\lim_{i,j\to\infty} L_\star(\zeta_i-\zeta_j)=L(0)$. If $\mu_u$--a.e. 
$\xi_j\to\xi$, then $\xi\in T^L_u\cE$, $\lim_{j\to\infty} L(\xi_j-\xi)=L(0)$ and $\lim_{j\to\infty} L(\xi_j)=L(\xi)$.
\end{lem}

The condition in (b) that $\xi_j$ converge almost everywhere often can be replaced by requiring
$\lim_{i,j\to\infty}L(\xi_i-\xi_j)=0$, and the existence of $\xi$ as in the conclusion still follows. For example, if $L$ is a norm on 
$T^L_u\cE$, then one can show that 
\begin{equation}             %9.5
L(\eta)\ge c\int_X|\eta| \,d\mu_u,\quad \eta\in T^L_u\cE,
\end{equation}
with some $c>0$. \cite[Lemma 6.1]{L1} gives this for $\eta\in T^\infty_u\cE$, but the same proof works for $\eta\in T^L_u\cE$. 
Alternatively, one can approximate $\eta\in T^L_u\cE$ by bounded $\eta_j\in T^\infty_u\cE$ and obtain (9.5) by dominated 
convergence,  Lemma 9.4. Once (9.5) is known, $L(\xi_i-\xi_j)\to 0$ implies $\xi_j$ is a Cauchy sequence in $T^1_u\cE$, hence 
converges to some $\xi\in T^1_u\cE$. Applying Lemma 9.7b to subsequences $\xi_{j_k}$ then gives 
$\xi\in T^L_u\cE$ and the rest.---In this case the norms $||\xi||_\beta=L_\star\big(\xi^\star(\beta\,\cdot)\big)$,
$0<\beta\le 1$, endow $T^L_u\cE$ with the structure of a metrizable topological vector space, which is 
complete (i.e., Fr\'echet) according to the above discussion.

\begin{proof}
Except for the very last limit in (b), $\lim_{j}L(\xi_j)=L(\xi)$, we can assume that $L$ is absolutely monotone and $L(0)=0$, in light of Theorem 8.6.

(a)\ Given $\varepsilon>0$, choose $j$ so 
 that $L\big(2(\xi_i-\xi_j)\big)<\varepsilon$ when $i>j$. Next choose $\delta>0$ so that $L(2\xi_i 1_E)<\varepsilon$ if 
 $i=1,2,\dots, j$ and $\mu_u(E)\le\delta$. With such $E$ and $i>j$
\[
2L(\xi_i 1_E)\le L\big(2(\xi_i-\xi_j) 1_E\big)+L(2\xi_j 1_E)<2\varepsilon.
\]
As $L(\xi_i 1_E)\le L(2\xi_i 1_E)<\varepsilon$ if $i\le j$, $L$ is indeed absolutely equicontinuous on $\{\xi_j: j\in\bN\}$.

(b)\ First observe that $\sup_j L(\xi_j)=M<\infty$: as before, we choose $j$ so that 
$L_\star \big(2(\xi^\star_i-\xi^\star_j)\big)<1$ when $i>j$, then
\[
2L(\xi_i)=2L_\star(\xi_i^\star)\le L_\star(2\xi_j^\star)+L_\star\big(2(\xi_i^\star-\xi_j^\star)\big)<L(2\xi_j)+1,\quad i>j.
\]

This also implies $L(\xi)\le M$ as follows. If $k\in\bN$, let $\xi_{jk}=\min(k,|\xi_j|)\in T^\infty_u\cE$. Then $\mu_u$--a.e. 
$\lim_j\xi_{jk}=\min(k,|\xi|)$, and by strong continuity
\[
L(\min(k,|\xi|))=\lim_{j\to\infty} L(\xi_{jk})\le M.
\]
We represent $L$ with a suitable $\cA_u\subset \bR\times B(X)$ as 
\[
L(\eta)=\sup_{(a,f)\in\cA_u} a+\int_X|\eta f| \,d\mu_u, \qquad\eta\in T^1\cE.
\]
If $(a,f)\in\cA_u$, by monotone convergence
\[
a+\int_X|\xi f|\,d\mu_u=\lim_{k\to\infty} a+\int_X\min (k, |\xi|) |f| \,d\mu_u\le\limsup_{k\to\infty} L\big(\min(k,|\xi|)\big)\le M,
\]
and so $L(\xi)\le M$.

Next, let $\theta:(X,\mu_u)\to \big((0,V),$ Lebesgue\big) preserve measure. By part (a), $L$ is absolutely equicontinuous on $\{\xi^\star_j\circ\theta, j\in\bN\}$, whence also on $\{\xi_j: j\in\bN\}$. Therefore the computation above, with $\xi$ replaced by $\xi 1_E$ gives that $L$ is absolutely continuous on $\xi$. But we can also replace $\xi$ by any of its allies $\eta$, to conclude that $\xi\in T^L_u\cE$, indeed.

To prove the two limits in part (b), we no longer assume $L$ is absolutely monotone, but it is convenient to keep the assumption $L(0)=0$. As before, with $\alpha\in\bR$ we let
\[
\xi_j(\alpha)=(1-\alpha)\xi+\alpha\xi_j\in T^L_u\cE,
\]
and estimate $L\big(\xi_j(\alpha)\big)$.
For fixed $\alpha$, by part (a), $L$ and $L^{|\, |}$ are absolutely equicontinuous on the family 
$\{\xi_j(\alpha)^\star\circ\theta: j\in\bN\}$, hence on $\{\xi_j(\alpha): j\in\bN\}$. Choose $\delta>0$ so that 
$L^{|\, |}\big(2\xi_j(\alpha) 1_E\big)< 1$ when $\mu_u(E)\le\delta$; and choose such an $E$ outside which 
$\xi_j\to\xi$ uniformly. Then, by (8.17) and by convexity, 
\begin{align*}
4|L\big(\xi_j(\alpha)\big)|&\le 4L^{|\, |}\big(\xi_j(\alpha)\big)\\
&\le 2L^{|\, |}\big(2\xi_j(\alpha)1_E\big)+L^{|\, |}\big(4(\xi_j(\alpha)-\xi)1_{X\setminus E}\big)+L^{|\, |}(4\xi 1_{X\setminus E}).
\end{align*}
As in the proof of Lemma 9.5, the first term on the right is $<2$, the last is $\le L^{|\, |}(4\xi)<\infty$, and since 
$\xi_j(\alpha)\to\xi$ uniformly on $X\setminus E$, the middle term goes to 0 as $j\to\infty$. Therefore 
$\limsup_j|L(\xi_j(\alpha))|\le 2+L^{|\, |}(4\xi)$, and Lemma 8.5 implies $\lim_j L(\xi_j)=L(\xi)$. Finally we
note that since $\xi\in T^L_u\cE$, $L$ is uniformly equicontinuous on $\{\alpha(\xi_j-\xi):j\in\bN\}$ for fixed 
$\alpha\in\bR$. Therefore in the above estimate we can
replace $\xi_j$ by $\xi_j-\xi\in T^L_u\cE$, to obtain $\lim_j L(\xi_j-\xi)=0$.
\end{proof}

\section{Rise, energy, action}      %Section 10

In this section we fix a strict rearrangement invariant Lagrangian $L:T^1\cE\to(-\infty,\infty]$ that is convex and lower semicontinuous on the fibers  $T^1_u\cE$. We assume that $L$ is finite and absolutely continuous on 
elements of $T^\infty\cE$, i.e., $T^\infty\cE\subset T^L\cE$, cf. Definitions 9.1--3 and Lemma 9.6. 
Generalizing Guedj--Zeriahi's high energy classes and Mabuchi's and Darvas's metrics [Da1--2, M, GZ1--2], 
we will introduce energy classes $\cE^L\subset \cE$ of $\omega$--plurisubharmonic functions and define 
action between elements of $\cE^L$. We will show that this action can be computed from the rise, and as an 
application, we prove a Principle of Least Action in $\cE^L$.

Recall that $u\in\cE$ belongs to the energy space $\cE^1$ if $u\in L^1(X,\mu_u)$, or $u\in T^1_u\cE$. 

\begin{defn}       %10.1
The energy space\footnote{The reader will notice an inconsistency in our notation, and should keep it in mind to avoid confusion. $\cE^1, \cE^\infty$ are not $\cE^L$ when $L\equiv 1$, resp. $\infty$. Instead, $\cE^1$ corresponds to $L(\xi)=\int_X|\xi|\,d\mu_u$, while the space $\cE^\infty$ is not of the type introduced here,
because the Lagrangian $L(\xi)= \text{ess\,sup}_X |\xi|$ is not allowed in this section. The same inconsistency already occurs in the notation $T^L\cE$.} $\cE^L$ consists of $u\in\cE^1$ that, viewed as elements of $T^1_u\cE$, are in $T^L_u\cE$.
%consists of $u\in\cE$ that, viewed as elements of $T^1_u\cE$ are in $T^L_u\cE$.
\end{defn}

\begin{lem}      %10.2
A function $u\in\cE^1$ is in $\cE^L$ if and only if $\rho[u,v]\in\cE^L(0,V)$ for some (equivalently: for every) $v\in \cE^\infty$.
\end{lem}

Recall the definition (9.4) of the function $L_\star: L^1(0,V)\to (-\infty,\infty]$ and of the associated space 
$\cE^L(0,V)\subset L^1(0,V)$ in the paragraph after. Lemma 9.4 implies $\cE^L(0,V)$ is a vector space of functions, contains $L^\infty(0,V)$, and if $\zeta\in L^1(0,V)$, $\zeta', \zeta''\in\cE^L(0,V)$, and 
$\zeta'\le\zeta\le\zeta''$ a.e., then $\zeta, |\zeta|\in\cE^L(0,V)$. If $\zeta,\zeta'\in L^1(0,V)$ are equidistributed,
or $\zeta'^\star$ is the $(\alpha, \beta)$ rescaling of $\zeta^\star$, and $\zeta\in\cE^L(0,V)$, then 
$\zeta'\in\cE^L(0,V)$.

\begin{proof}
Suppose first $u\in\cE^L$. Thus $u$ and so  $-u\in T^L_u\cE$, whence $(-u)^{\star u}\in\cE^L(0,V)$. 
If $v\in\cE^\infty$, then $(v-u)^{\star u}=(-u)^{\star u}+O(1)\in\cE^L(0,V)$. With $c=\inf(v-u)>-\infty$ 
\[
c\le\rho[u,v]\le (v-u)^{\star u}
\]
by Theorem 6.1, and so $\rho[u,v]\in \cE^L(0,V)$.

Conversely, suppose $\rho[u,v]\in\cE^L(0,V)$, where $v\in\cE^\infty$. With $c$ above, $u+c\le v$. 
Taking into account that $\mu_{u+c}=\mu_u$, Theorem 6.1 implies for $0<s<V$
\begin{equation}         %10.1
\rho[u+c, v](s)\le (v-u-c)^{\star u}(s)\le (n+1)\rho [u+c, v](s/e).
\end{equation}
Here $\rho[u+c,v]=\rho[u,v]-c\in\cE^L(0,V)$, see Lemma 5.10. If $\xi\in\cE^L$ is chosen so that 
$\xi^{\star u}=\rho[u+c, v]$, then the function on the right of (10.1) is the decreasing rearrangement of some 
$\eta\in\cE^L$ allied with $\xi$, Definition 9.2. Thus both lower and upper estimates in (10.1) are in 
$\cE^L(0,V)$, hence so must be $(v-u-c)^{\star u}$. In other words $v-u-c\in T_u^L\cE$. Since
$v-c$ is bounded, $\pm u\in T^L_u\cE$ and  $u\in\cE^L$ follow.
\end{proof}
\begin{lem}           %10.3
If $u\in\cE^L$, $v\in\cE$ and $u\le v$, then $v\in\cE^L$.
\end{lem}

\begin{proof}
This follows from Lemmas 5.10, 10.2. If $c=\max v$, then $-c=\rho[c,0]\le\rho[v,0]\le\rho[u,0]$. Since $-c,\rho[u,0]\in\cE^L(0,V)$, 
also $\rho[v,0]\in\cE^L(0,V)$.
\end{proof}

Consider now $u,v\in\cE^L$ and pick a constant $c\ge u,v$. By Lemma 5.10
\begin{equation}         %10.2
\rho[c, v]\le \rho[u,v]\le \rho [u,c].
\end{equation}
Since $\rho[c,v]$ is the decreasing rearrangement of $-\rho[v,c]\in\cE^L(0,V)$, with respect to Lebesgue measure, the two extremes in (10.2) are in $\cE^L(0,V)$, hence so is $\rho[u,v]$.
In particular, $L_\star(\lambda\rho[u,v])<\infty$ with any $\lambda\in\bR$.

\begin{defn}         %10.4
If $T>0$, the action between $u,v\in\cE^L$ is
\[\cL_T(u,v)=TL_\star(\rho[u,v]/T)<\infty.\]
\end{defn}

\cite{L2} already defined action between $u,v\in\cE^\infty$, see (2.10). That the current definition is 
consistent with (2.10) will follow from the next result, in which $\cL_T$ is used in the sense of Definition 10.4.

\begin{lem} %10.5
If $u_j,v_j\in\cE^L$ decrease to $u,v\in\cE^L$, then
\begin{equation}      %10.3
\cL_T(u,v)=\lim_{j\to\infty}\cL_T(u_j, v_j)
\end{equation}
\end{lem}

\begin{proof}
By Theorem 5.6 $\rho[u_j, v_j]\to\rho[u,v]$ almost everywhere. With a constant $c\ge u_1,v_1$ we 
have $\rho[c,v]\le\rho[u_j,v_j]\le\rho[u,c]$, and dominated convergence (Lemma 9.5), together with (9.4) yields 
$L_\star(\rho[u_j, v_j]/T)\to L_\star(\rho[u,v]/T)$, i.e., (10.3).
\end{proof}
It follows that Definition 10.4 and (2.10) give the same notion of action on $\cE^\infty$. This is so
over $\cH^{1\bar1}$ by Lemma 3.4; in general we take $u_j,v_j\in\cH^{1\bar1}$ that decrease to $u,v\in\cE^\infty$,
and note that (10.3) holds with the definition (2.3) too, by \cite[Lemma 9.4]{L2}.

In the special case when $L$ is Orlicz norm associated with a convex, `normalized', even $\chi:\bR\to\bR$, as in 
\cite[section 1.1]{Da1}, 
\[
L(\xi)=\inf\Big\{r>0: \int_X\chi(\xi/r) \,d\mu_u\le \chi(1)\Big\}, \quad \xi\in T^1_u\cE,
\]
$\cL_T$ is independent of $T$ and coincides with Darvas's distance function $d_\chi$ if $\chi$ is in one of the classes 
$\cW^+_p$, $1\le p<\infty$. That 
\begin{equation} %10.4
d_\chi(u,v)=\cL_T(u,v)
\end{equation}
first follows from \cite[Theorem 1]{Da1} when $u,v\in\cH$; and for general $u,v\in\cE^L$ from Darvas's definition \cite[(5)]{Da1} 
of $d_\chi$ as a limit, and from (10.3).

A version of the triangle inequality holds for general $L$:
\begin{equation}       %10.5
\cL_S(u,v)+\cL_T(v,w)\ge \cL_{S+T}(u,w),\qquad u,v,w\in\cE^L.
\end{equation}
When $u,v,w\in\cE^\infty$, \cite[(5.3)]{L2} implies with any $a\in\bR$, e.g., 
\[
\cL_S(u,v)=\inf\int^{a+S}_a L(\partial_t\psi(t))dt,
\]
the infimum taken over all piecewise $C^1$ paths $\psi:[a,a+S]\to \cE^\infty\subset B(X)$; then (10.5) follows by concatenating 
paths. To general $u,v,w\in \cE^L$ we can decrease by $u_j, v_j, w_j\in\cE^\infty$. Then 
$\cL_S(u_j,v_j)+\cL_T(v_j, u_j)\ge\cL_{S+T}(u_j, w_j)$, 
and (10.5) follows by passing to the limit, cf. (10.3).

The notion of action of a path and the principle of least action of \cite{L2} can be extended from $\cE^\infty$ to $\cE^L$. Let $\varphi:[a,b]\to\cE^L$ be an arbitrary map. Its action is 
\[
\cL(\varphi)=\sup\sum^m_{i=1}\cL_{t_i-t_{i-1}}\big(\varphi(t_{i-1}),\varphi(t_i)\big)\le\infty,
\]
the supremum taken over all partitions $a=t_0<t_1<\ldots<t_m=b$. \cite[Theorem 10.1]{L2} identifies this quantity with 
$\int^b_a L(\partial_t\varphi(t))dt$ when $\varphi$ is a $C^1$-map into $\cE^\infty\subset B(X)$.

\begin{thm}[Principle of Least Action]          %10.6
If $\varphi:[a,b]\to\cE$ is a geodesic and $\varphi(a),\varphi(b)\in\cE^L$, then $\varphi(t)\in\cE^L$ for all $t\in(a,b)$. If, furthermore, $\psi:[a,b]\to\cE^L$ is any map with $\psi(a)=\varphi(a)$, $\psi(b)=\varphi(b)$, then $\cL(\psi)\ge\cL(\varphi)$. 
\end{thm}

This generalizes parts of \cite[Theorem 6]{Da2}, \cite[Theorem 4.11]{Da1}, and \cite[Theorem 1.1]{L2}. 

\begin{proof}
As in [Da1-2], we start by showing that $u,v\in\cE^L$ implies $u\wedge v\in\cE^L$. By
\cite[Corollary 3.5]{Da2} $u\wedge v\in \cE^1$. Choose a constant 
$c\ge u,v$. Lemma 6.5 implies 
\begin{equation*} 
0\le\rho[u\wedge v, c](s)\le(c-u)^{\star u}(s/2)+(c-v)^{\star v}( s/2),
\qquad 0<s<V.
\end{equation*}
Since the function on the right is in $\cE^L(0,V)$, as is $0$, so must be the function\
in the middle;  whence by Lemma 10.2  $u\wedge v\in\cE^L$.

Now with $\varphi$ of the theorem the constant path $\tva(t)\equiv \varphi(a)\wedge \varphi(b)$ is 
a subgeodesic, and  $\tva(a)\le \varphi(a)$, $\tva(b)\le \varphi(b)$. Therefore
\[
\varphi(a)\wedge\varphi(b)\le \varphi(t),\qquad t\in[a,b],
\]
and $\varphi(a)\wedge\varphi(b)\in\cE^L$, whence by Lemma 10.3 $\varphi(t)\in\cE^L$.

If $a\le t<t'\le b$ then
\[
\frac{\rho[\varphi(t), \varphi(t')]}{t'-t}=\rho_{\varphi|[t,t']}=\rho_\varphi=\frac{\rho[\varphi(a), \varphi(b)]}{b-a}.
\]
Hence, with any partition $a=t_0<t_1<\dots<t_m=b$
\begin{multline*}
\sum^m_{i=1}\cL_{t_i-t_{i-1}}\big(\varphi(t_{i-1}),\varphi(t_i)\big)=
\sum^m_{i=1}(t_i-t_{i-1})L_\star\Big(\frac{\rho[\varphi(t_{i-1}), \varphi(t_i)]}{t_i-t_{i-1}}\Big)\\
=\sum^m_{i=1}(t_i-t_{i-1})L_\star(\rho_\varphi)=(b-a)L_\star(\rho_\varphi)=\cL_{b-a}\big(\varphi(a),\varphi(b)\big).
\end{multline*}
Therefore $\cL(\varphi)=\cL_{b-a}\big(\varphi(a),\varphi(b)\big)$. As for the action of $\psi$, 
\[
\sum^m_{i=1}\cL_{t_i-t_{i-1}}\big(\psi(t_{i-1}),\psi(t_i)\big)\ge\cL_{b-a}\big(\psi(a), \psi(b)\big)=
\cL_{b-a}\big(\varphi(a), \varphi(b)\big)
\]
by the triangle inequality (10.4). Thus $\cL(\psi)\ge\cL_{b-a}\big(\varphi(a), \varphi(b)\big)=\cL(\varphi)$, q.e.d.
\end{proof}

\end{document}